\documentclass[11pt]{amsart}
\usepackage{graphicx}
\usepackage[mathcal]{euscript}
\usepackage{float}

\restylefloat{figure}

\newtheorem{theorem}{Theorem}[section]
\newtheorem{lemma}[theorem]{Lemma}
\newtheorem{corollary}[theorem]{Corollary}
\newtheorem{proposition}[theorem]{Proposition}

\theoremstyle{definition}

\theoremstyle{remark}
\newtheorem{remark}[theorem]{Remark}

\numberwithin{equation}{section}

\newcommand{\bb}{{\boldsymbol{b}}}

\def\wt{\widetilde}
\def\bb{\mathbb}

\begin{document}

\title{The Khovanov width of twisted links and closed $3$-braids}

\author[A. Lowrance]{Adam Lowrance}

\address{Department of Mathematics, Louisiana State University \newline
\hspace*{.375in} Baton Rouge, LA 70817, USA} \email{\rm{lowrance@math.lsu.edu}}

\subjclass{}
\date{January 15, 2009}

\begin{abstract}
Khovanov homology is a bigraded $\bb{Z}$-module that categorifies the Jones polynomial. The support of Khovanov homology lies on a finite number of slope two lines with respect to the bigrading. The Khovanov width is essentially the largest horizontal distance between two such lines. We show that it is possible to generate infinite families of links with the same Khovanov width from link diagrams satisfying certain conditions. Consequently, we compute the Khovanov width for all closed 3-braids.
\end{abstract}

\maketitle

\section{Introduction}

Let $L\subset S^3$ be an oriented link. The Khovanov homology of $L$, denoted $Kh(L)$, was introduced by Mikhail Khovanov in \cite{kho}, and is a bigraded $\bb{Z}$-module with homological grading $i$ and polynomial (or Jones) grading $j$ so that $Kh(L) = \bigoplus_{i,j}Kh^{i,j}(L)$. The graded Euler characteristic of $Kh(L)$ is the unnormalized Jones polynomial:
$$(q+q^{-1}) V_L(q^2) = \sum_{i,j}(-1)^i\text{rank}~Kh^{i,j}(L)q^j.$$

The support of $Kh(L)$ lies on a finite number of slope 2 lines with respect to the bigrading. Therefore, 
it is convenient to define the $\delta$-grading by $\delta=j- 2i$ so that $Kh(L)=\bigoplus_\delta Kh^\delta(L)$. Also, either all the $\delta$-gradings of $Kh(L)$ are odd, or they all are even. Let $\delta_{\text{min}}$ be the minimum $\delta$-grading where $Kh(L)$ is nontrivial and $\delta_{\text{max}}$ be the maximum $\delta$-grading where $Kh(L)$ is nontrivial. Then $Kh(L)$ is said to be {\it $[\delta_{\text{min}},\delta_{\text{max}}]$-thick}, and  the {\it Khovanov width of $L$} is defined as $$w_{Kh}(L)=\frac{1}{2}(\delta_{\text{max}}-\delta_{\text{min}})+1.$$

In this paper, we show the following:
\begin{itemize}
\item If a crossing in a link diagram is width-preserving (defined in Section \ref{twistedlinks}), then it can be replaced with an alternating rational tangle and the Khovanov width does not change (Theorem \ref{alt-tangle}).
\item We compute the Khovanov width of all closed 3-braids (Theorem \ref{3width}).
\item We determine the Turaev genus of all closed 3-braids, up to an additive error of at most 1.
\item We show that for closed 3-braids the Khovanov width and odd Khovanov width are equal (Corollary \ref{oddwidth}).
\end{itemize}

The paper is organized as follows. In Section \ref{background}, we review some properties of Khovanov homology. In Section \ref{twistedlinks}, we describe the behavior of Khovanov width when a crossing is replaced by an alternating rational tangle. In Section \ref{3-braid sec}, the Khovanov width of any closed 3-braid is computed. Finally, in Section \ref{oddsec} we show that Khovanov width and odd Khovanov width for closed 3-braids are equal.

\newpage

\noindent{\bf Acknowledgements.} I wish to thank Scott Baldridge, Oliver Dasbach, Mikhail Khovanov, and Peter Ozsv\'ath  for many helpful conversations. A portion of the work for this article was done while the author was visiting Columbia University in the Fall of 2008. He thanks the department of mathematics for their hospitality.

\section{Khovanov homology background}
\label{background}

In this section, we give background material on Khovanov homology. If $D$ is a diagram for $L$, then denote the Khovanov homology of $L$ by either $Kh(L)$ or $Kh(D)$. Similarly, let $w_{Kh}(L)$ and $w_{Kh}(D)$ equivalently denote the Khovanov width of $L$. If $\bb{F}$ is a field, then let $Kh(L;\bb{F})$ denote $Kh(L)\otimes\bb{F}$ and $w_{Kh}(L;\bb{F})$ denote the width of $Kh(L;\bb{F})$. 

Let $L_1$ and $L_2$ be oriented links, and let $C$ be a component of $L_1$. Denote by $l$ the linking number of $C$ with its complement $L_1-C$. Let $L_1^\prime$ be the link $L_1$ with the orientation of $C$ reversed. Denote the mirror image of $L_1$ by $\overline{L_1}$ and the disjoint union of $L_1$ and $L_2$ by $L_1\sqcup L_2$. The following proposition was proved by Khovanov in \cite{kho}.
\begin{proposition}[Khovanov]
\label{kholink}
For $i,j\in\bb{Z}$ there are isomorphisms
\begin{eqnarray*}
Kh^{i,j}(L_1^\prime) & \cong & Kh^{i+2l,j+2l}(L_1),\\
Kh^{i,j}(\overline{L_1};\bb{Q}) & \cong & Kh^{-i,-j}(L_1;\bb{Q})\\
\text{Tor}(Kh^{i,j}(\overline{L_1})) & \cong & \text{Tor}(Kh^{1-i,-j}(L_1)),\text{ and}\\
Kh^{i,j}(L_1\sqcup L_2) & \cong & \oplus_{k,m\in\bb{Z}}(Kh^{k,m}(L_1) \otimes Kh^{i-k,j-m}(L_2))\oplus\\
& & \oplus_{k,m\in\bb{Z}}Tor^{\bb{Z}}_1(Kh^{k,m}(L_1),Kh^{i-k+1,j-m}(L_2))
\end{eqnarray*}
\end{proposition}
Let $D$ be a diagram for $L_1$ and $D^\prime$ be the diagram $D$ with the component $C$ reversed. Denote the number of negative crossings in $D$ by neg$(D)$, where the sign of a crossing is as in Figure \ref{orresol}. Set $s=\text{neg}(D)-\text{neg}(D^\prime)$. Then Proposition \ref{kholink} implies 
\begin{eqnarray*}
Kh^\delta(D^\prime)&\cong  & Kh^{\delta+s}(D),\text{ and}\\
Kh^{\delta}(\overline{L_1};\bb{Q}) & \cong & Kh^{-\delta}(L_1;\bb{Q}).
\end{eqnarray*}

In \cite{kho2}, Khovanov introduced the reduced Khovanov homology. For a knot $K$, this theory is denoted $\wt{Kh}(K)$. For links of more than one component, the reduced Khovanov homology depends on a choice of a marked component, and hence is denoted $\wt{Kh}(L,C)$, where $C$ is the marked component of $L$. Similar to the unreduced version, $\wt{Kh}(L,C)$ is a bigraded $\bb{Z}$-module with homological grading $i$ and Jones grading $j$ so that $\wt{Kh}(L,C)=\bigoplus_{i,j}\wt{Kh}^{i,j}(L,C)$. The graded Euler characteristic of $\wt{Kh}(L,C)$ is the ordinary Jones polynomial:
$$V_L(q^2) = \sum_{i,j}(-1)^i\text{rank}~\wt{Kh}^{i,j}(L,C)q^j.$$
As with Khovanov homology, if $\wt{\delta}_{\text{min}}$ is the minimum $\delta$-grading where $\wt{Kh}(L,C)$ is nontrivial and $\wt{\delta}_{\text{max}}$ is the maximum $\delta$-grading where $\wt{Kh}(L,C)$ is nontrivial, then we say that $\wt{Kh}(L,C)$ is $[\wt{\delta}_{\text{min}},\wt{\delta}_{\text{max}}]$-thick. The reduced Khovanov width is defined as $w_{\wt{Kh}}(L) = \frac{1}{2}(\wt{\delta}_{\text{max}}-\wt{\delta}_{\text{min}})+1$. 

Asaeda and Przytycki \cite{ap} show that there is a long exact sequence relating reduced and unreduced Khovanov homology. 
\begin{theorem}[Asaeda-Przytycki]
\label{untored}
There is a long exact sequence relating the reduced and unreduced versions of Khovanov homology:
$$\cdots\to\wt{Kh}^{i,j+1}(L,C)\to Kh^{i,j}(L)\to\wt{Kh}^{i,j-1}(L,C)\to\wt{Kh}^{i+1,j+1}(L,C)\to\cdots$$
\end{theorem}

\begin{corollary}
\label{redwid}
Let $L$ be a link with marked component $C$. Then $Kh(L)$ is $[\delta_{\text{min}},\delta_{\text{max}}]$-thick if and only if  $\wt{Kh}(L,C)$ is $[\delta_{\text{min}}+1,\delta_{\text{max}}-1]$-thick. Hence $w_{Kh}(L) - 1 = w_{\wt{Kh}}(L)$.
\end{corollary}
\begin{proof}
The long exact sequence of Theorem \ref{untored} can be rewritten with respect to the $\delta$-grading as
$$\cdots\to\wt{Kh}^{\delta+1}(L,C)\to Kh^\delta(L)\to\wt{Kh}^{\delta-1}(L,C)\to\wt{Kh}^{\delta-1}(L,C)\to\cdots.$$
Suppose $Kh(L)$ is $[\delta_{\text{min}},\delta_{\text{max}}]$-thick. Therefore $\wt{Kh}^\delta(L,C)=0$ for $\delta>\delta_{\text{max}} +1$ and for $\delta< \delta_{\text{min}} - 1$. 

Suppose $\wt{Kh}^{\delta_{\text{max}}+1}(L,C)$ is nontrivial. Then for some $i$ and $j$ where $j-2i=\delta_{\text{max}}+1$, the group $\wt{Kh}^{i,j}(L,C)$ is nontrivial. By repeatedly applying the long exact sequence of Theorem \ref{untored}, one sees that $\wt{Kh}^{i+k,j+2k}(L,C)$ is nontrivial for all $k\geq 0$. However, the group $\wt{Kh}^{\delta_{\text{max}}+1}(L,C)$ is finitely generated. Hence $\wt{Kh}^{\delta_{\text{max}}+1}(L,C)$ is trivial. Similarly, one can show that $\wt{Kh}^{\delta_{\text{min}}-1}(L,C)$ is also trivial.

The long exact sequence also implies that $\wt{Kh}^{\delta_{\text{max}}-1}(L,C)$ and $\wt{Kh}^{\delta_{\text{min}}+1}(L,C)$ are nontrivial. Thus $\wt{Kh}(L,C)$ is $[\delta_{\text{min}}+1,\delta_{\text{max}}-1]$-thick.

Suppose $\wt{Kh}(L,C)$ is $[\delta_{\text{min}}+1,\delta_{\text{max}}-1]$-thick. Similar to the case above, if either $Kh^{\delta_{\text{min}}}(L)$ or $Kh^{\delta_{\text{max}}}(L)$ are trivial, then one can show that $\wt{Kh}^{\delta_{\text{min}}+1}(L,C)$ or $\wt{Kh}^{\delta_{\text{max}}-1}(L,C)$ respectively are infinitely generated. Hence $Kh(L)$ is $[\delta_{\text{min}},\delta_{\text{max}}]$-thick.
\end{proof}
Corollary \ref{redwid} implies that if $C$ and $C^\prime$ are two components of $L$, then $\wt{Kh}(L,C)$ is $[\wt{\delta}_{\text{min}},\wt{\delta}_{\text{max}}]$-thick if and only if $\wt{Kh}(L,C^\prime)$ is $[\wt{\delta}_{\text{min}},\wt{\delta}_{\text{max}}]$-thick. Hence, the notation $w_{\wt{Kh}}(L)$ is unambiguous.

\begin{figure}[h]
\includegraphics[scale = 0.4]{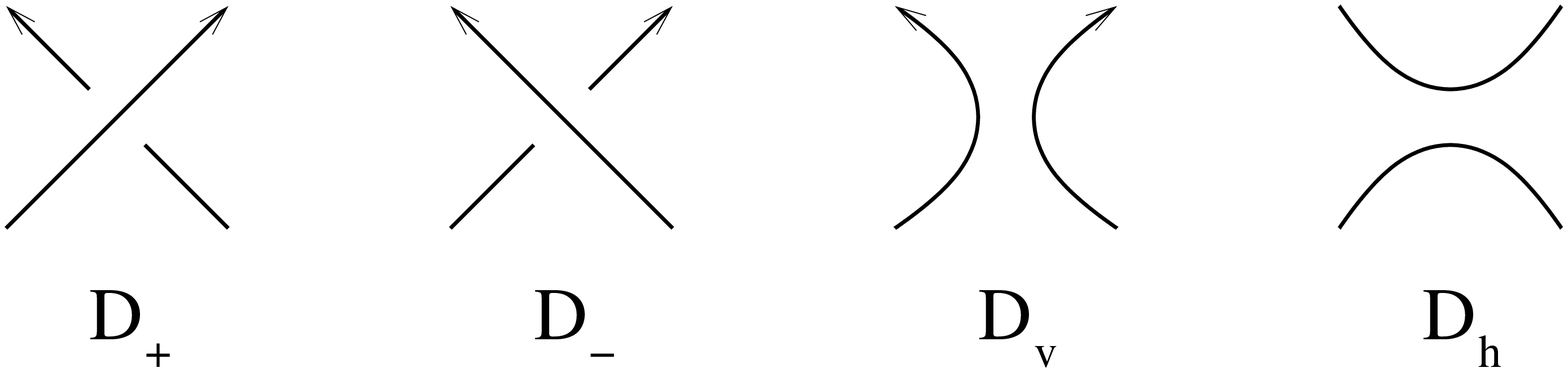}
\caption{The links in an oriented resolution. $D_+$ is a positive crossing, and $D_-$ is a negative crossing.}
\label{orresol}
\end{figure}

Let $D_+, D_-, D_v$ and $D_h$ be planar diagrams of links that agree outside a neighborhood of a distinguished crossing $x$ as in Figure \ref{orresol}. Define $e=\text{neg}(D_h)-\text{neg}(D_+)$. There are long exact sequences relating the Khovanov homology of each of these links. Khovanov \cite{kho} implicitly describes these sequences, and Viro \cite{vir} explicitly states both sequences. The graded versions are taken from Rasmussen \cite{ras} and Manolescu-Ozsv\'{a}th \cite{mo}.
\begin{theorem}[Khovanov]
\label{les}
There are long exact sequences
$$\cdots \to Kh^{i-e-1,j-3e-2}(D_h) \xrightarrow{}Kh^{i,j}(D_+) \xrightarrow{}Kh^{i,j-1}(D_v) \xrightarrow{}Kh^{i-e,j-3e-2}(D_h) \to\cdots$$
and
$$\cdots \to Kh^{i,j+1}(D_v) \xrightarrow{}Kh^{i,j}(D_-) \xrightarrow{}Kh^{i-e+1,j-3e+2}(D_h) \xrightarrow{}Kh^{i+1,j+1}(D_v) \to\cdots.$$
\end{theorem}

When only the $\delta = j - 2i$ grading is considered, the long exact sequences become
$$\cdots \to Kh^{\delta - e}(D_h) \xrightarrow{f^{\delta-e}_+}Kh^\delta(D_+) \xrightarrow{g^\delta_+}Kh^{\delta-1}(D_v) \xrightarrow{h^{\delta-1}_+}Kh^{\delta-e-2}(D_h) \to\cdots$$
and
$$\cdots \to Kh^{\delta+1}(D_v) \xrightarrow{f^{\delta+1}_-}Kh^\delta(D_-) \xrightarrow{g^\delta_-}Kh^{\delta-e}(D_h) \xrightarrow{h^{\delta-e}_-}Kh^{\delta-1}(D_v) \to \cdots$$
There are versions of these long exact sequences where Khovanov homology is replaced with reduced Khovanov homology. In the reduced sequences, the gradings are identical to the unreduced sequences. 

\begin{figure}[h]
\includegraphics[scale=.4]{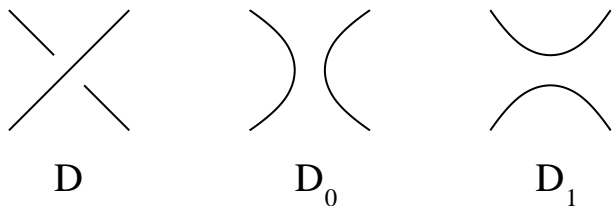}
\caption{The links in an unoriented resolution.}
\label{unresol}
\end{figure}

Let $D$, $D_0$ and $D_1$ be link diagrams differing only in a neighborhood of a crossing $x$ of $D$ (as in Figure \ref{unresol}) with associated link types $L$, $L_0$ and $L_1$ respectively.
The set $\mathcal{Q}$ of {\it quasi-alternating links} is the smallest set of links such that
\begin{itemize}
\item The unknot is in $\mathcal{Q}$.
\item If the link $L$ has a diagram with a crossing $x$ such that
\begin{enumerate}
\item both of the links, $L_0$ and $L_1$ are in $\mathcal{Q}$,
\item det$(L)=$ det$(L_0) +$det$(L_1),$
\end{enumerate}
\end{itemize}
then $L$ is in $\mathcal{Q}$.
We will say that $D$ is {\it quasi-alternating at $x$}. 

In \cite{lee}, Lee showed that alternating links have reduced Khovanov width 1. The set of alternating links is a proper subset of the set of quasi-alternating links. Manolescu and Ozsv\'{a}th \cite{mo} use the long exact sequences for reduced Khovanov homology to show that the same result holds for quasi-alternating links.
\begin{theorem}[Manolescu-Ozsv\'{a}th]
\label{qalt}
Let $L$ be a quasi-alternating link. Then $\wt{Kh}(L)$ is supported entirely in $\delta$-grading $-\sigma(L)$, where $\sigma(L)$ denotes the signature of the link.
\end{theorem}
Theorem \ref{qalt} together with Corollary \ref{redwid} imply that if $L$ is quasi-alternating, then $Kh(L)$ is $[-\sigma(L)-1,- \sigma(L)+1]$-thick and $w_{Kh}(L)=2$.
\medskip

Theorem \ref{les} directly implies the following corollary:
\begin{corollary}
\label{glue}
Let  $D_+,D_-,D_v$ and $D_h$ be as in Figure \ref{orresol}. Suppose $Kh(D_v)$ is $[v_{\text{min}},v_{\text{max}}]$-thick and $Kh(D_h)$ is $[h_{\text{min}},h_{\text{max}}]$-thick. Then $Kh(D_+)$ is $[\delta_{\text{min}}^+,\delta_{\text{max}}^+]$-thick, and $Kh(D_-)$ is  $[\delta_{\text{min}}^-,\delta_{\text{max}}^-]$-thick , where
\begin{equation*}
\delta^+_{\text{min}} =
\begin{cases}
\min\{v_{\text{min}} +1, h_{\text{min}} +e\} & \text{if $v_{\text{min}} \neq h_{\text{min}} + e+1$}\\
v_{\text{min}} + 1 & \text{if $v_{\text{min}}=h_{\text{min}}+e+1$ and $h^{v_{\text{min}}}_+$ is surjective}\\
v_{\text{min}} - 1 & \text{if $v_{\text{min}}=h_{\text{min}}+e+1$ and $ h^{v_{\text{min}}}_+$ is not surjective,}
\end{cases}
\end{equation*}
\begin{equation*}
\delta^+_{\text{max}} =
\begin{cases}
\max\{v_{\text{max}} +1, h_{\text{max}} +e\} & \text{if $v_{\text{min}} \neq h_{\text{max}} + e+1$}\\
v_{\text{max}} - 1 & \text{if $v_{\text{max}}=h_{\text{max}}+e+1$ and $h^{v_{\text{max}}}_+$ is injective}\\
v_{\text{max}} + 1& \text{if $v_{\text{max}}=h_{\text{max}}+e+1$ and $ h^{v_{\text{max}}}_+$ is not injectve,}
\end{cases}
\end{equation*}
\begin{equation*}
\delta^-_{\text{min}} =
\begin{cases}
\min\{v_{\text{min}} -1, h_{\text{min}} +e\} & \text{if $v_{\text{min}} \neq h_{\text{min}} + e-1$}\\
v_{\text{min}} + 1 & \text{if $v_{\text{min}}=h_{\text{min}}+e-1$ and $h^{v_{\text{min}}}_-$ is surjective}\\
v_{\text{min}} - 1 & \text{if $v_{\text{min}}=h_{\text{min}}+e-1$ and $ h^{v_{\text{min}}}_-$ is not surjective,}
\end{cases}
\end{equation*}
and
\begin{equation*}
\delta^-_{\text{max}} =
\begin{cases}
\max\{v_{\text{max}} -1, h_{\text{max}} +e\} & \text{if $v_{\text{max}} \neq h_{\text{max}} + e-1$}\\
v_{\text{max}} - 1 & \text{if $v_{\text{max}}=h_{\text{max}}+e-1$ and $h^{v_{\text{max}}}_+$ is injective}\\
v_{\text{max}} + 1 & \text{if $v_{\text{max}}=h_{\text{max}}+e-1$ and $ h^{v_{\text{max}}}_+$ is not injective.}
\end{cases}
\end{equation*}
\end{corollary}

\section{Twisted Links}
\label{twistedlinks}

\subsection{Khovanov width of twisted links}

Let $\tau=C(a_1,\dots,a_m)$ be a rational tangle, and let $D$ be a link diagram with a distinguished crossing $x$. Suppose the slopes of the arcs near $x$ are $\pm 1$. Define {\it $D$ twisted at $x$ by $\tau$}  to be the diagram obtained by removing $x$ and inserting $\tau$ such that a neighborhood of the rightmost crossing or topmost crossing of $\tau$ in $D_\tau$ looks exactly like a neighborhood of $x$ in $D$. The resulting link diagram is denoted $D_\tau$. See Figure \ref{twist}.
\begin{figure}[h]
\includegraphics[scale=.25]{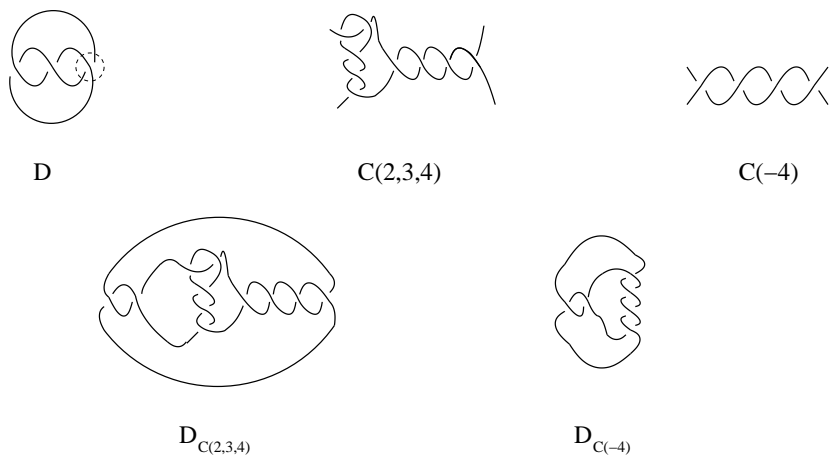}
\caption{The diagram $D$ twisted by $C(2,3,4)$ and $C(-4)$.}
\label{twist}
\end{figure}

The main result of this section, Theorem \ref{alt-tangle}, is a generalization of a proposition proved by Champanerkar and Kofman in \cite{cha-kof}.
\begin{proposition}[Champanerkar-Kofman]
Let $D$ be a link diagram with crossing $x$, and let $\tau$ be an alternating rational tangle such that $D$ is twisted at $x$ by $\tau$. If $D$ is quasi-alternating at $x$, then $D_\tau$ is quasi-alternating at each crossing of $\tau$.
\end{proposition}

Let $D$ be a diagram with crossing $x$. Resolve $D$ at the crossing $x$ to obtain diagrams $D_v$ and $D_h$. Suppose $Kh(D_v)$ is $[v_{\text{min}},v_{\text{max}}]$-thick and $Kh(D_h)$ is $[h_{\text{min}},h_{\text{max}}]$-thick. As before, set $e=\text{neg}(D_h) - \text{neg}(D_+)$, where $D_+$ is the same diagram as $D$ except if the crossing $x$ in $D$ is negative, then it is changed to positive in $D_+$. The diagram $D$ is said to be width-preserving at $x$ if either of the following conditions hold.
\begin{itemize}
\item If $x$ is a positive crossing in $D$, then both $v_{\text{min}}\neq h_{\text{min}} + e +1$ and $v_{\text{max}}\neq  h_{\text{max}}+e+1$.
\item If $x$ is a negative crossing in $D$, then both $v_{\text{min}}\neq h_{\text{min}} + e -1$ and $v_{\text{max}}\neq h_{\text{max}}+ e -1$.
\end{itemize}
\begin{proposition}
\label{qaltgen}
Let $D$ be a link diagram with crossing $x$. If $D$ is quasi-alternating at $x$, then $D$ is width-preserving at $x$.
\end{proposition}
\begin{proof}
Suppose $D$ is quasi-alternating at $x$.  Let $D_v$ and $D_h$ be the two resolutions of $D$ at $x$. Since $D$ is quasi-alternating at $x$, it follows that $D_v$ and $D_h$ are also quasi-alternating. Theorem \ref{qalt} implies that $\wt{Kh}(D)$, $\wt{Kh}(D_v)$ and $\wt{Kh}(D_h)$ are each supported entirely in one $\delta$-grading. Suppose $\wt{Kh}(D_v)$ is supported in $\delta$-grading $v$ and $\wt{Kh}(D_h)$ is supported in $\delta$-grading $h$. Corollary \ref{untored} implies that $Kh(D_v)$ is $[v-1,v+1]$-thick and $Kh(D_h)$ is $[h-1,h+1]$-thick. Let $e=\text{neg}(D_h)-\text{neg}(D_+)$ where $D_+$ is the same diagram as $D$ except if $x$ is negative in $D$, then it is changed to positive in $D_+$. Since $\text{det}(D)=\text{det}(D_v)+\text{det}(D_h)$, it follows that the nontrivial parts of $\wt{Kh}(D)$, $\wt{Kh}(D_v)$ and $\wt{Kh}(D_h)$ lie in three consecutive spots in the long exact sequence of Theorem \ref{les} such that $\wt{Kh}(D_v)$ and $\wt{Kh}(D_h)$ are not adjacent. Therefore, if $x$ is positive, then $v=h+e-1$, and if $x$ is negative, then $v=h+e+1$. The result follows directly.
\end{proof}

\begin{lemma} 
\label{step}
Let $D$ be an oriented link diagram with crossing $x$, and let $\tau$ be an alternating rational tangle with exactly two crossings $x_0$ and $x_1$.  Let $D_\tau$ be $D$ twisted at $x$ by $\tau$. If $D$ is width-preserving at $x$, then for any orientation, $D_\tau$ is width-preserving at $x_0$ and $x_1$. Moreover, $w_{Kh}(D)=w_{Kh}(D_\tau)$.
\end{lemma}
\begin{proof} 
There are two ways to twist $D$ at $c$, either horizontally or vertically. Let $\tau_1=C(2)$ and $\tau_2=C(-2)$. 
\begin{figure}
\end{figure}
For each case, it is only necessary to prove the result for one choice of orientations on $D$ and $D_\tau$. Proposition \ref{kholink} implies the result for all other choices of orientations on $D$ and $D_\tau$. 

Let $D_v$ and $D_h$ be the diagrams obtained by resolving $D$ at $x$, and let $D^i_v$ and $D^i_h$ be the diagrams obtained by resolving $D_\tau$ at the crossing $x_i$ for $i=0,1$. Suppose $Kh(D_v)$ and $Kh(D_h)$ are $[v_{\text{min}},v_{\text{max}}]$-thick and $[h_{\text{min}},h_{\text{max}}]$-thick respectively. Let $e=\text{neg}(D_h) - \text{neg}(D_+)$ where $D_+$ is the same diagram as $D$ except if the crossing $x$ is negative in $D$, then it is changed to positive in $D_+$. Similarly set $e_i=\text{neg}(D^i_h)-\text{neg}(D^i_+)$ where $D^i_+$ is the same diagram as $D_\tau$ except if the crossing $x_i$ is negative in $D_\tau$, then it is changed to positive in $D_+$.

\begin{figure}[h]
\includegraphics[scale=.3]{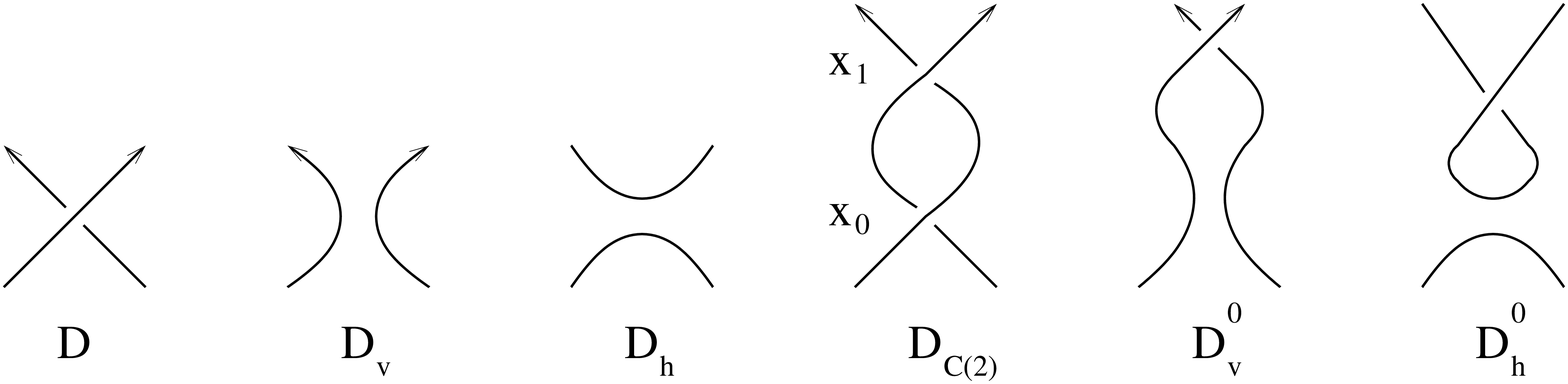}
\caption{The resolutions for $x$ positive and $\tau=C(2)$.}
\label{xor}
\end{figure}
Suppose $x$ is positive. Choose the orientation on $D_{\tau_1}$ given in Figure \ref{xor}. Also, Figure \ref{xor} shows the resolutions $D^0_v$ and $D^0_h$. 

Observe that $x_i$ is positive in $D_{\tau_1}$ for $i=0,1$.
Corollary \ref{glue} implies that $Kh(D)$ is $[\alpha,\beta]$-thick where $\alpha=\min\{v_{\text{min}}+1,h_{\text{min}}+e\}$ and $\beta=\max\{v_{\text{max}}+1,h_{\text{max}}+e\}$. The diagrams $D^i_v$ and $D$ represent the same link, and the diagrams $D_h$ and $D^i_h$ represent the same link. Therefore, $Kh(D^i_v)$ is $[\alpha,\beta]$-thick and $Kh(D^i_h)$ is $[h_{\text{min}},h_{\text{max}}]$-thick. The diagram $D^i_h$ is the same as the diagram $D_h$ except $D^i_h$ has one additional negative Reidemeister I twist, and hence $\text{neg}(D^i_h)=\text{neg}(D_h)+1$. Since the diagrams $D$ and $D^i_v$ are identical, $\text{neg}(D)=\text{neg}(D^i_v)$. Thus $e_i=e+1$.
Since $D$ is width-preserving, it follows that $v_{\text{min}}\neq h_{\text{min}} +e +1$ and $v_{\text{max}}\neq h_{\text{max}}+ e + 1$. Therefore,
$$h_{\text{min}}+e_i+1 = h_{\text{min}} + e +2 \neq \alpha,$$
and
$$h_{\text{max}}+e_i+1 = h_{\text{max}} + e + 2 \neq \beta.$$
Hence $D_{\tau_1}$ is width-preserving at $x_i$. Also, Corollary \ref{glue} implies that $Kh(D_{\tau_1})$ is $[\alpha+1,\beta+1]$-thick, and thus $w_{Kh}(D) = w_{Kh}(D_{\tau_1})$.

\begin{figure}[h]
\includegraphics[scale=.3]{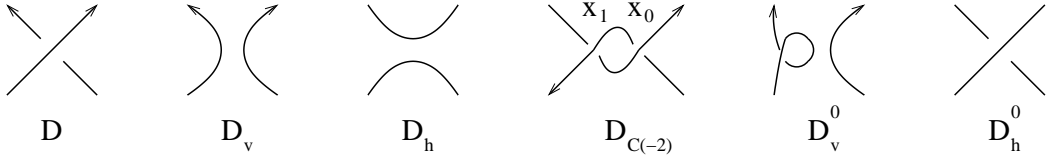}
\caption{The resolutions for $x$ positive, $\tau=C(-2)$, and with the depicted strands of $D$ in the same component.}
\label{xunsame}
\end{figure}

The possible orientations of $D_{\tau_2}$ depend on whether the strands forming the crossing $x$ are in the same component of $D$ or different components of $D$. Suppose they are in the same component. Choose the orientation on $D_{\tau_2}$ given in Figure \ref{xunsame}. Also, Figure \ref{xunsame} shows the resolutions $D^0_v$ and $D^0_h$. 

Observe that $x_i$ is positive in $D_{\tau_2}$ for $i=0,1$.
With suitably chosen orientations, we have
\begin{equation}
\label{samecomp1}
\text{neg}(D_v)=\text{neg}(D)=\text{neg}(D^i_h),
\end{equation}
and
\begin{equation}
\label{samecomp2}
\text{neg}(D^i_+)=\text{neg}(D_h).
\end{equation}
The diagram $D^i_v$ is the same as $D_v$ except $D^i_v$ has one component reversed and an additional positive Reidemeister I twist. Therefore, Proposition \ref{kholink} implies that $Kh(D^i_v)$ is $[v_{\text{min}}-e,v_{\text{max}}-e]$-thick. Also, equations \ref{samecomp1} and \ref{samecomp2} imply that $e_i=-e$. The diagram $D^i_h$ is identical to $D$. Therefore, $Kh(D^i_h)$ is $[\alpha,\beta]$-thick where $\alpha=\min\{v_{\text{min}}+1,h_{\text{min}}+e\}$ and $\beta=\max\{v_{\text{max}}+1,h_{\text{max}}+e\}$.
Since $D$ is width-preserving at $x$, we have $v_{\text{min}}\neq h_{\text{min}}+e+1$ and $v_{\text{max}}\neq h_{\text{max}}+e+1$. Therefore,
$$\alpha+e_i+1 = \min\{v_{\text{min}}+1,h_{\text{min}}+e\}-e +1 = \min\{v_{\text{min}}-e+2,h_{\text{min}}+1\}\neq v_{\text{min}}-e,$$
and
$$\beta + e_i +1 = \max\{v_{\text{max}}+1,h_{\text{max}}+e\} -e +1 = \max\{v_{\text{max}}-e+2,h_{\text{max}}+1\}\neq v_{\text{max}}-e.$$
Thus $D_{\tau_2}$ is width-preserving at $x_i$. Moreover, Corollary \ref{glue} implies that $Kh(D_{\tau_2})$ is \newline$[\alpha-e,\beta-e]$-thick, and hence $w_{Kh}(D)=w_{Kh}(D_{\tau_2}).$

\begin{figure}[h]
\includegraphics[scale=.3]{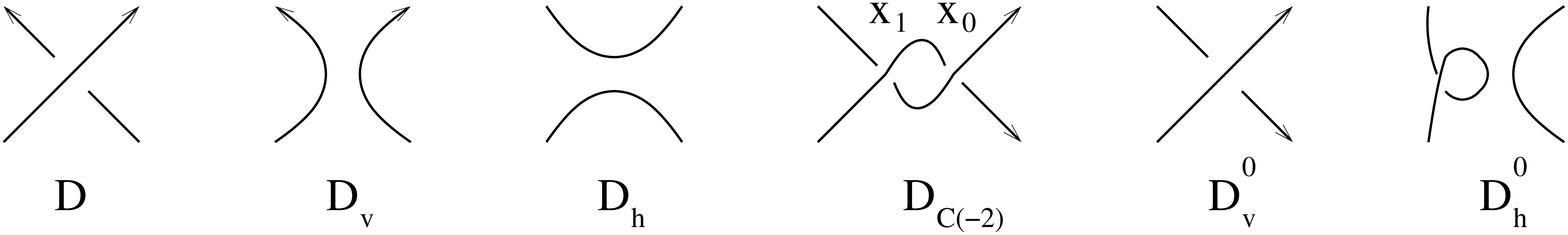}
\caption{The resolutions for $x$ positive, $\tau=C(-2)$, and with the depicted strands of $D$ in different components.}
\label{xundiff}
\end{figure}
Suppose the strands forming the crossing $x$ are in different components of the link. Choose the orientation on $D_{\tau_2}$ given in Figure \ref{xundiff}. Also, Figure \ref{xundiff} shows the resolutions $D^0_v$ and $D^0_h$. 
\begin{figure}
\end{figure}
Observe that $x_i$ is a negative crossing in $D_{\tau_2}$ for $i=0,1$. Orient $D^i_h$ so that it represents the same oriented link as $D_v$. With a suitably chosen orientation on $D_h$, we have
\begin{equation}
\label{diffcomp1}
\text{neg}(D)=\text{neg}(D_v)=\text{neg}(D^i_h),
\end{equation}
and
\begin{equation}
\label{diffcomp2}
\text{neg}(D_h)+1=\text{neg}(D^i_+)=\text{neg}(D^i_v).
\end{equation}
Equations \ref{diffcomp1} and \ref{diffcomp2} imply that $e_i=-e-1$.
The diagram $D^i_v$ is the same as $D$ except $D^i_v$ has one component reversed. Equations \ref{diffcomp1} and \ref{diffcomp2} along with Proposition \ref{kholink} imply that $Kh(D^i_v)$ is $[\alpha-e-1,\beta-e-1]$-thick where $\alpha=\min\{v_{\text{min}}+1,h_{\text{min}}+e\}$ and $\beta=\max\{v_{\text{max}}+1,h_{\text{max}}+e\}$. Since $D^i_h$ and $D_v$ represent the same oriented link, it follows that $Kh(D^i_h)$ is $[v_{\text{min}},v_{\text{max}}]$-thick. Since $D$ is width-preserving at $x$, we have $v_{\text{min}}\neq h_{\text{min}}+ e +1$ and $v_{\text{max}}\neq h_{\text{max}}+e +1$. Therefore,
$$\alpha -e -1 = \min\{v_{\text{min}}-e,h_{\text{min}}-1\}\neq v_{\text{min}}-e-2= v_{\text{min}}+e_i-1,$$
and
$$\beta-e-1 = \max\{v_{\text{max}}-e,h_{\text{max}}-1\}\neq v_{\text{max}}-e-2 = v_{\text{max}}+e_i-1.$$
Thus $D_{\tau_2}$ is width-preserving at $x_i$. Moreover, Corollary \ref{glue} implies that $Kh(D_{\tau_2})$ is $[\alpha-e-1,\newline \beta-e-1]$-thick, and hence $w_{Kh}(D)=w_{Kh}(D_{\tau_2})$,

The case where $x$ is a negative crossing in $D$ is proved similarly.
\end{proof}

\begin{theorem}
\label{alt-tangle}
Let $D$ be a link diagram with crossing $x$, $\tau$ be an alternating rational tangle, and $D_\tau$ be the diagram $D$ twisted at $x$ by $\tau$. If $D$ is width-preserving at $x$, then $w_{Kh}(D) = w_{Kh}(D_\tau)$.
\end{theorem}
\begin{proof}
Let $\tau=C(a_1,\dots,a_m)$. Since $\tau$ is alternating, either $a_i>0$ for all $i$ or $a_i<0$ for all $i$. Suppose $a_i>0$ for all $i$. Beginning with the diagram $D$ and the crossing $x$, one can alternate twisting the diagram by $C(2)$ and $C(-2)$. Replacing the appropriate crossings $m$ times results in the diagram $D_{\tau^\prime}$ where $\tau^\prime=C(2,1,\dots,1)$.  Lemma \ref{step} implies that each crossing in $D_{\tau^\prime}$ is width-preserving, and $w_{Kh}(D)=w_{Kh}(D_{\tau^\prime})$.

Replace crossings corresponding to the $m$-th term in $\tau^\prime$ by $C(2)$ until the resulting diagram is obtained by twisting $D$ by $C(2,1,\dots,1,a_m)$ at $x$. Next, replace crossings corresponding to the $(m-1)$-st term in $C(2,1,\dots,1,a_m)$ with $C(-2)$ until the resulting diagram is obtained by twisting $D$ by $C(2,1,\dots,1,a_{m-1},a_m)$ at $x$. Continue replacing crossings in the tangle by either $C(2)$ or $C(-2)$ until the resulting diagram is obtained by twisting $D$ by $C(a_1,\dots,a_m)$ at $x$.
Since at each step, the only tangles used are $C(2)$ and $C(-2)$, Lemma \ref{step} implies that $w_{Kh}(D)=w_{Kh}(D_\tau)$. The case where each $a_i<0$ is proved similarly.
\end{proof}

\begin{figure}[h]
\includegraphics[scale=.18]{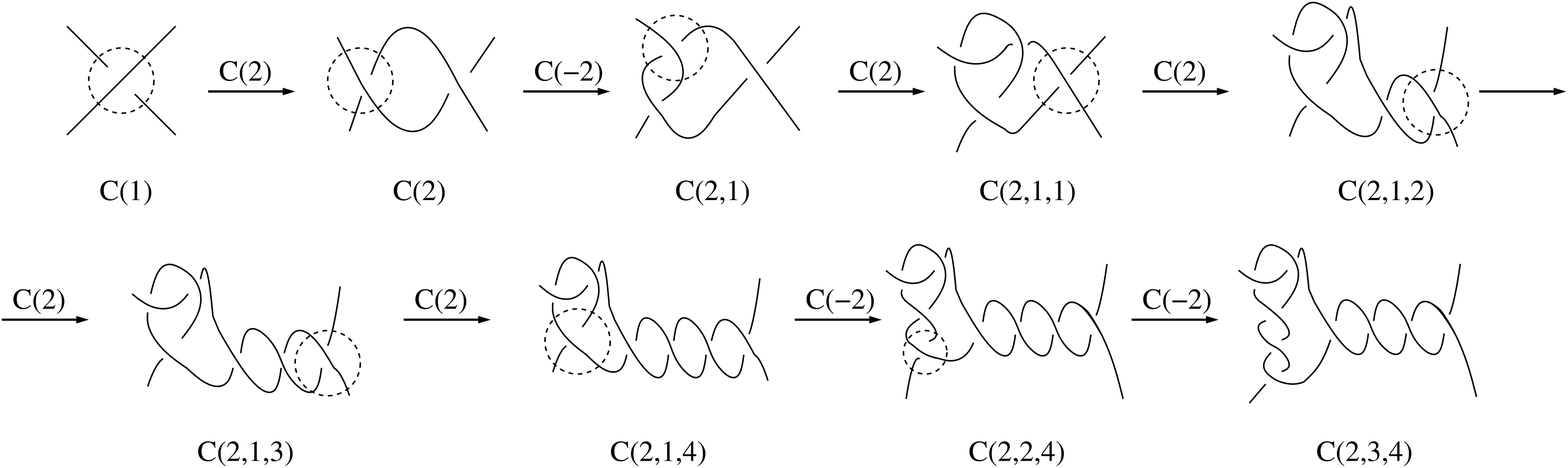}
\caption{The inductive process of Theorem \ref{alt-tangle}. At each step, the circled crossing is replaced with either $C(2)$ or $C(-2)$.}
\label{induct}
\end{figure}

\begin{remark}
Watson \cite{wat} proves that $w_{Kh}(D_\tau)$ is bounded by $w_{Kh}(D_v)$ and $w_{Kh}(D_h)$. By assuming that $D$ is width-preserving at $x$, we are able to strengthen the result and calculate $w_{Kh}(D_\tau)$.
\end{remark}

Suppose $D$ is an oriented diagram with crossing $x$. If $D$ is twisted at $x$ by $\tau_n=C(n)$ as in Figure \ref{braidtwist}, then the assumptions of Theorem \ref{alt-tangle} can be relaxed and a slightly stronger result holds. The following technical result is needed to compute the Khovanov width of closed 3-braids.
\begin{figure}[h]
\includegraphics[scale=.35]{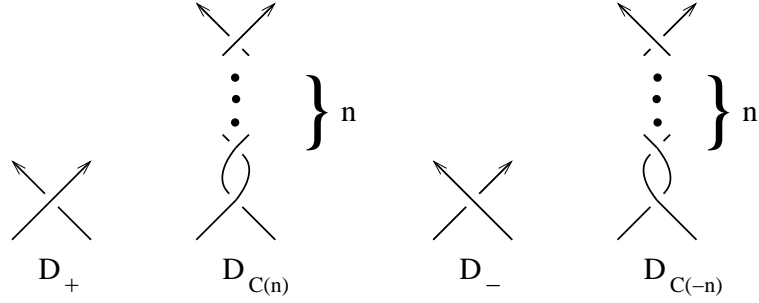}
\caption{For $n>0$, twist $D_+$ by $C(n)$ and twist $D_-$ by $C(-n)$. Then choose the above orientations for $D_{C(n)}$ and $D_{C(-n)}$.}
\label{braidtwist}
\end{figure}
\begin{figure}
\end{figure}
\begin{proposition}
\label{band}
Suppose $D$ is an oriented diagram with crossing $x$. Suppose $D$ is twisted at $x$ by $\tau_n=C(n)$ as in Figure \ref{braidtwist}. Let $D_v$ and $D_h$ be the two resolutions of $D$ at $x$. Suppose $Kh(D_v)$ is $[v_{\text{min}},v_{\text{max}}]$-thick and $Kh(D_h)$ is $[h_{\text{min}},h_{\text{max}}]$-thick. Let $\alpha_{\pm} = \min\{v_{\text{min}}\pm 1,h_{\text{min}}+e\}$ and $\beta_{\pm}=\max\{v_{\text{max}}\pm 1, h_{\text{max}}+e\}$.
\begin{enumerate}
\item Let $n>0$. Suppose that $v_{\text{min}}\neq h_{\text{min}} + e +1$. If $v_{\text{max}}=h_{\text{max}}+e+1$, then suppose that there exist integers $i$ and $j$ such that $j-2i=v_{\text{max}}$, $Kh^{i,j}(D_v)$ is nontrivial, and $Kh^{k,l}(D_h)$ is trivial for all $k$ whenever $l\leq j-3e-1$. Then $Kh(D_{\tau_n})$ is $[n+\alpha_+,n+\beta_+]$-thick.
\item Let $n<0$. Suppose that $v_{\text{max}}\neq h_{\text{max}} +e -1$. If $v_{\text{min}}=h_{\text{min}} + e -1$, then suppose that there exist integers $i$ and $j$ such that $j-2i=v_{\text{min}}$, $Kh^{i,j}(D_v)$ is nontrivial, and $Kh^{k,l}(D_h)$ is trivial for all $k$ whenever $l\geq j-3e-1$. Then $Kh(D_{\tau_n})$ is $[n+\alpha_-,n+\beta_-]$-thick.
\end{enumerate}
\end{proposition}
\begin{proof}
Let $n>0$. Since $D$ is twisted at $x$ by $\tau_n$ as in Figure \ref{braidtwist}, it follows that $x$ is a positive crossing. If both $v_{\text{min}}\neq h_{\text{min}}+ e +1$ and $v_{\text{max}}\neq h_{\text{max}}+e +1$, then $D$ is width-preserving at $x$. It follows from the proof of Theorem \ref{alt-tangle} that $Kh(D_{\tau_n})$ is $[n+\alpha_+, n+\beta_+]$-thick.

Suppose $v_{\text{min}}\neq h_{\text{min}} +e + 1$ and $v_{\text{max}} = h_{\text{max}}+e+1$. Thus there exist integers $i$ and $j$ such that $j-2i=v_{\text{max}}$, $Kh^{i,j}(D_v)$ is nontrivial, and $Kh^{k,l}(D_h)$ is trivial for all $k$ and for all $l\leq j-3e-1$. Since $v_{\text{min}}\neq h_{\text{min}}+e+1$, it follows that the minimum $\delta$-grading where $Kh(D_{\tau_n})$ is nontrivial is $n+\alpha_+$. We show, by induction on $n$, that $Kh^{i,j+n}(D_{\tau_n})\cong Kh^{i,j}(D_v)$.  This implies that the maximum $\delta$-grading supporting $Kh(D_{\tau_n})$ is $n+\beta_+$.

If $n=1$, then the long exact sequence of Theorem \ref{les} looks like
$$0\to Kh^{i,j+1}(D)\to Kh^{i,j}(D_v)\to Kh^{i-e,j-3e-1}(D_h)\to\cdots.$$
By hypothesis, $Kh^{i-e,j-3e-1}(D_h)$ is trivial, and hence $Kh^{i,j+1}(D)\cong Kh^{i,j}(D_v)$.

Suppose, by way of induction, that $Kh^{i,j+n}(D_{\tau_n})\cong Kh^{i,j}(D_v)$. Resolve $D_{\tau_{n+1}}$ at any crossing in $\tau_{n+1}$ to obtain diagrams $D_v^\prime$ and $D_h^\prime$. Let $e_{n+1}=\text{neg}(D_h^\prime)-\text{neg}(D_{\tau_{n+1}})$. Since $\text{neg}(D_h^\prime)=\text{neg}(D_h)+n$ and $\text{neg}(D_{\tau_{n+1}})=\text{neg}(D)$, it follows that $e_{n+1}=e+n$. Observe that $D_v^\prime$ and $D_{\tau_{n}}$ are the same diagram, and $D_h^\prime$ and $D_h$ are diagrams for the same link. Hence the long exact sequence of Theorem \ref{les} looks like
$$0\to Kh^{i,j+n+1}(D_{\tau_{n+1}})\to Kh^{i,j+n}(D_{\tau_n})\to Kh^{i-e-n,j-3e-3n-1}(D_h)\to\cdots.$$
Since $j-3e-3n-1\leq j - 3e -1$, it follows that $Kh^{i-e-n,j-3e-3n-1}(D_h)$ is trivial. Thus $Kh^{i,j+n+1}(D_{\tau_{n+1}})\cong Kh^{i,j+n}(D_{\tau_n})\cong Kh^{i,j}(D_v)$. Therefore $Kh(D_{\tau_n})$ is $[n+\alpha_+,n+\beta_+]$-thick.

The case where $n<0$ is proved in a similar fashion using the second sequence from Theorem \ref{les}.
\end{proof}

\subsection{The Turaev genus of twisted links}

Each link diagram $D$ has an associated Turaev surface $\Sigma_D$. Let $\Gamma$ be the plane graph associated to $D$. Regard $\Gamma$ as embedded in $\mathbb{R}^2$ sitting inside $\mathbb{R}^3$. Outside the neighborhoods of the vertices of $\Gamma$ is a collection of arcs in the plane. Replace each arc by a band that is perpendicular to the plane. In the neighborhoods of the vertices, place a saddle so that the circles obtained from choosing a $0$-resolution at each crossing lie above the plane and so that the circles obtained from choosing a $1$-resolution at each crossing lie below the plane (see Figure \ref{saddle}).
\begin{figure}[h]
\includegraphics[scale=.6]{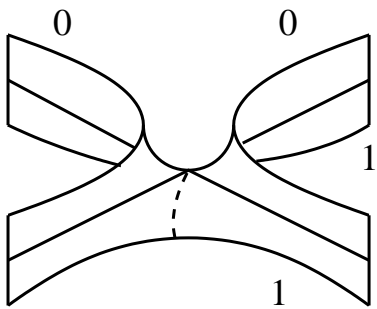}
\caption{In a neighborhood of each crossing, insert a saddle so that the boundary above the plane corresponds to the $0$ resolution and the boundary below the plane corresponds to the $1$ resolution.}
\label{saddle}
\end{figure}
The resulting surface has a boundary of disjoint circles, with circles corresponding to the all $0$-resolution above the plane and circles corresponding to the all $1$-resolution below the plane. For each boundary circle, insert a disk to obtain a closed surface $\Sigma_D$ known as the {\it Turaev surface} (cf. \cite{tur}). The genus of this surface is denoted $g(\Sigma_D)$, and can be calculated by the formula
$$g(\Sigma_D) = \frac{2-s_0(D)-s_1(D)+c(D)}{2},$$
where $c(D)$ is the number of crossings in $D$ and $s_0(D)$ and $s_1(D)$ are the number of circles appearing in the all $0$ and all $1$ resolutions of $D$ respectively.  The {\it Turaev genus} of a link is defined as
$$g_T(L) = \min \{g(\Sigma_D)~|~D~\text{is a diagram for}~L\}.$$

The Turaev genus of a link $L$ is a measure of how far $L$ is away from being alternating. Specifically, Dasbach et. al. \cite{das} prove the following proposition.
\begin{proposition}[Dasbach-Futer-Kalfagianni-Lin-Stoltzfus]
A link has Turaev genus $0$ if and only if it is alternating.
\end{proposition}

Also, the Turaev genus of $L$ gives a bound on the Khovanov width of $L$. Manturov \cite{man} and Champanerkar-Kofman-Stoltzfus \cite{stoltz} prove the following inequality.
\begin{proposition}[Manturov, Champanerkar-Kofman-Stoltzfus]
\label{bound}
Let $L$ be a link. Then $$w_{Kh}(L)-2\leq g_T(L).$$ 
\end{proposition}

The following proposition is implicit in Champanerkar and Kofman \cite{cha-kof}, but not explicitly proven.
\begin{proposition}
\label{turaevtwist}
Let $D$ be a link diagram with crossing $x$, and let $\tau$ be an alternating rational tangle such that $D$ is twisted by $\tau$ at $x$. Then $g(\Sigma_{D_\tau})=g(\Sigma_D)$.
\end{proposition}
\begin{proof}
Suppose $\tau=C(a_1,\dots,a_m)$, where $\text{sign}(a_i)=\text{sign}(a_j)$ for all $i$ and $j$. Let $a=\sum_{i=1}^m |a_i|$. The all 0-resolution of $D$ is the same as the all 0-resolution of $D_\tau$, except $D_\tau$ has an additional $k$ circles. Similarly, the all 1-resolution of $D$ is the same as the all 1-resolution of $D_\tau$, except $D_\tau$ has an additional $l$ circles. Since $\tau$ is alternating, it follows that $k+l=a-1$. Also, $c(D_\tau)=c(D)+a-1$. Therefore,
\begin{eqnarray*}
g(\Sigma_D)  & = & \frac{2-s_0(D)-s_1(D)+c(D)}{2} \\
& = & \frac{2 - (s_0(D_\tau) + s_1(D_\tau) -(a-1))+c(D_\tau)- (a-1)}{2}\\
& = & \frac{2-s_0(D_\tau)-s_1(D_\tau)+c(D_\tau)}{2}\\
& = & g(\Sigma_{D_\tau}).
\end{eqnarray*}
\end{proof}

In the case where $D$ is the closure of a braid, there is a particularly nice version of Proposition \ref{turaevtwist}. Let $w=w(\sigma_1,\sigma_1^{-1},\dots,\sigma_{n-1},\sigma_{n-1}^{-1}) \in B_n$ be a word in the braid group, and let $D$ be the link diagram obtained from taking the closure of $w$. Suppose $w^\prime$ is word in $B_n$ obtained by replacing $\sigma_i$ in $w$ with $\sigma_i^k$ where $k>0$ or by replacing $\sigma_i^{-1}$ in $w$ with $\sigma_i^k$ where $k<0$. Let $D^\prime$ be the link diagram obtained by taking the braid closure of $w^\prime$.
\begin{corollary}
\label{braidreduc}
Let $D$ and $D^\prime$ be link diagrams obtained from the closures of the braids $w$ and $w^\prime$ respectively. Then $g(\Sigma_D) = g(\Sigma_{D^\prime})$.
\end{corollary}

\section{Applications to 3-braids}
\label{3-braid sec}

Closed 3-braids are a rich class of links in which computation of invariants are possible. In \cite{birman}, Birman and Menasco classify the link types of closed 3-braids. Several papers (Schreier \cite{sch}, Murasugi \cite{mur}, and Garside \cite{gar}) give algorithms to determine when two 3-braids are conjugate in $B_3$. In this paper, we will be interested in Murasugi's solution to the conjugacy problem.
\subsection{Torus Links}

Let $T(p,q)$ denote the $(p,q)$ torus link.  In this subsection, we will determine the Turaev genus and Khovanov width of $T(3,q)$. Turner \cite{turn} and Sto\v{s}i\'{c} \cite{stos} give formulas for the rational Khovanov homology of $T(3,q)$. The following theorem specifies the support of $Kh(T(3,q);\bb{Q})$ for $q\geq 3$. If $q\leq -3$, one can deduce the support from this theorem and the fact that $T(3,-q)$ is the mirror of  $T(3,q)$. 
\begin{theorem}[Sto\v{s}i\'{c}, Turner] 
\label{turn} Suppose $n\geq 1$.
\begin{enumerate}
\item The group $Kh(T(3,3n);\bb{Q})$ is $[4n-3,6n-1]$-thick. Thus $w_{Kh}(T(3,3n);\bb{Q}) = n + 2.$
\item The group $Kh(T(3,3n+1);\bb{Q})$ is $[4n-1,6n+1]$-thick. Thus $w_{Kh}(T(3,3n+1);\bb{Q}) = n + 2.$
\item The group $Kh(T(3,3n+2);\bb{Q})$ is $[4n+1,6n+3]$-thick. Thus $w_{Kh}(T(3,3n+2);\bb{Q}) = n + 2.$
\end{enumerate}
\end{theorem}

The following lemma gives several normal forms for braids in $B_3$ whose closures are torus links. We will use these normal forms to compute the Turaev genus of a $(3,q)$ torus link as well as the Turaev genus of many closed 3-braids.
\begin{lemma}
\label{norm}
Let $B_3$ be the braid group on three strands. Then for any $n>1$, we have
\begin{eqnarray*}
(\sigma_1\sigma_2)^3 & = & \sigma_1^2\sigma_2\sigma_1^2\sigma_2,\\
(\sigma_1\sigma_2)^4 & = & \sigma_1^2\sigma_2\sigma_1^3\sigma_2\sigma_1,\\
(\sigma_1\sigma_2)^5 & = & \sigma_1^3\sigma_2\sigma_1^3\sigma_2\sigma_1^2,\\
(\sigma_1\sigma_2)^{3n} & = & \sigma_1^{3}\sigma_2
\underbrace{\sigma_1^4\sigma_2\cdots\sigma_1^4\sigma_2}_{n-2}\sigma_1^3\sigma_2\sigma_1^{n+1}\sigma_2,\\
(\sigma_1\sigma_2)^{3n+1} & = & \sigma_1^{3}\sigma_2
\underbrace{\sigma_1^4\sigma_2\cdots\sigma_1^4\sigma_2}_{n-2}\sigma_1^3\sigma_2\sigma_1^{n+2}\sigma_2\sigma_1,\text{ and}\\
(\sigma_1\sigma_2)^{3n+2} & =  & \sigma_1^{3}\sigma_2
\underbrace{\sigma_1^4\sigma_2\cdots\sigma_1^4\sigma_2}_{n-1}\sigma_1^3\sigma_2\sigma_1^{n+1}
\end{eqnarray*}
\end{lemma}
\begin{proof}
Observe
\begin{eqnarray*}
(\sigma_1\sigma_2)^3 & = & \sigma_1\sigma_2\sigma_1\sigma_2\sigma_1\sigma_2 \\
& = & \sigma_1^2\sigma_2\sigma_1^2\sigma_2,\\
(\sigma_1\sigma_2)^4 & = & \sigma_1^2\sigma_2\sigma_1^2\sigma_2\sigma_1\sigma_2\\
& = & \sigma_1\sigma_2^2\sigma_1\sigma_2^3\sigma_1,\text{ and}\\
(\sigma_1\sigma_2)^5 & = & \sigma_1\sigma_2\sigma_1\sigma_2\sigma_1\sigma_2\sigma_1\sigma_2\sigma_1\sigma_2\\
& = & \sigma_1\sigma_2\sigma_1\sigma_2^2\sigma_1\sigma_2\sigma_1\sigma_2\sigma_1\\
& = & \sigma_1^2\sigma_2\sigma_1\sigma_2\sigma_1^2\sigma_2\sigma_1^2\\
& = & \sigma_1^3\sigma_2\sigma_1^3\sigma_2\sigma_1^2.
\end{eqnarray*}

The braid relation directly implies the following two relations:
\begin{eqnarray*}
\sigma_1^k\sigma_2\sigma_1 & = & \sigma_2\sigma_1\sigma_2^k,~\text{and}\\
\sigma_1\sigma_2\sigma_1^k & = & \sigma_2^k\sigma_1\sigma_2,
\end{eqnarray*}
for $k>0$. These relations will be used to prove the last three equations in the lemma.

For $n>1$, we prove that
$$(\sigma_1\sigma_2)^{3n} =  \sigma_1^{3}\sigma_2
\underbrace{\sigma_1^4\sigma_2\cdots\sigma_1^4\sigma_2}_{n-2}\sigma_1^3\sigma_2\sigma_1^{n+1}\sigma_2$$
by induction. Let $n=2$. Then
\begin{eqnarray*}
(\sigma_1\sigma_2)^6 & = & \sigma_1\sigma_2\sigma_1\sigma_2\sigma_1\sigma_2\sigma_1\sigma_2\sigma_1\sigma_2\sigma_1\sigma_2\\
& = & \sigma_1^2\sigma_2\sigma_1\sigma_2\sigma_1\sigma_2\sigma_1\sigma_2\sigma_1^2\sigma_2\\
& = & \sigma_1^3\sigma_2\sigma_1^3\sigma_2\sigma_1^3\sigma_2.
\end{eqnarray*}

Suppose, by way of induction, that
$$(\sigma_1\sigma_2)^{3n}=\sigma_1^{3}\sigma_2
\underbrace{\sigma_1^4\sigma_2\cdots\sigma_1^4\sigma_2}_{n-2}\sigma_1^3\sigma_2\sigma_1^{n+1}\sigma_2.$$
Then
\begin{eqnarray*}
(\sigma_1\sigma_2)^{3(n+1)} & = & \sigma_1^{3}\sigma_2
\underbrace{\sigma_1^4\sigma_2\cdots\sigma_1^4\sigma_2}_{n-2}\sigma_1^3\sigma_2\sigma_1^{n+1}\sigma_2\sigma_1\sigma_2\sigma_1\sigma_2\sigma_1\sigma_2\\
& = & \sigma_1^{3}\sigma_2
\underbrace{\sigma_1^4\sigma_2\cdots\sigma_1^4\sigma_2}_{n-2}\sigma_1^3\sigma_2\sigma_1^{n+1}\sigma_2\sigma_1\sigma_2^2\sigma_1\sigma_2^2\\
& = & \sigma_1^{3}\sigma_2
\underbrace{\sigma_1^4\sigma_2\cdots\sigma_1^4\sigma_2}_{n-2}\sigma_1^3\sigma_2\sigma_1^{n+3}\sigma_2\sigma_1^2\sigma_2^2\\
& = & \sigma_1^{3}\sigma_2
\underbrace{\sigma_1^4\sigma_2\cdots\sigma_1^4\sigma_2}_{n-2}\sigma_1^3\sigma_2\sigma_1\sigma_2\sigma_1\sigma_2^{n+2}\sigma_1\sigma_2^2\\
& = & \sigma_1^{3}\sigma_2
\underbrace{\sigma_1^4\sigma_2\cdots\sigma_1^4\sigma_2}_{n-2}\sigma_1^3\sigma_2\sigma_1\sigma_2\sigma_1^2\sigma_2\sigma_1^{n+2}\sigma_2\\
& = & \sigma_1^{3}\sigma_2
\underbrace{\sigma_1^4\sigma_2\cdots\sigma_1^4\sigma_2}_{n-2}\sigma_1^4\sigma_2\sigma_1^3\sigma_2\sigma_1^{n+2}\sigma_2.
\end{eqnarray*}
Hence, for all $n>1$,
\begin{equation}
\label{fulltwist}
(\sigma_1\sigma_2)^{3n} =  \sigma_1^{3}\sigma_2
\underbrace{\sigma_1^4\sigma_2\cdots\sigma_1^4\sigma_2}_{n-2}\sigma_1^3\sigma_2\sigma_1^{n+1}\sigma_2.
\end{equation}

Equation \ref{fulltwist} implies
\begin{eqnarray*}
(\sigma_1\sigma_2)^{3n+1} & = & \sigma_1^{3}\sigma_2
\underbrace{\sigma_1^4\sigma_2\cdots\sigma_1^4\sigma_2}_{n-2}\sigma_1^3\sigma_2\sigma_1^{n+1}\sigma_2\sigma_1\sigma_2 \\
& = & \sigma_1^{3}\sigma_2
\underbrace{\sigma_1^4\sigma_2\cdots\sigma_1^4\sigma_2}_{n-2}\sigma_1^3\sigma_2\sigma_1^{n+2}\sigma_2\sigma_1.
\end{eqnarray*}
Furthermore,
\begin{eqnarray*}
(\sigma_1\sigma_2)^{3n+2} & = & \sigma_1^{3}\sigma_2
\underbrace{\sigma_1^4\sigma_2\cdots\sigma_1^4\sigma_2}_{n-2}\sigma_1^3\sigma_2\sigma_1^{n+2}\sigma_2\sigma_1\sigma_1\sigma_2\\
& = &  \sigma_1^{3}\sigma_2
\underbrace{\sigma_1^4\sigma_2\cdots\sigma_1^4\sigma_2}_{n-2}\sigma_1^3\sigma_2\sigma_1\sigma_2\sigma_1\sigma_2^{n+1}\sigma_1\sigma_2\\
& = & \sigma_1^{3}\sigma_2
\underbrace{\sigma_1^4\sigma_2\cdots\sigma_1^4\sigma_2}_{n-2}\sigma_1^3\sigma_2\sigma_1\sigma_2\sigma_1^2\sigma_2\sigma_1^{n+1}\\
& = & \sigma_1^{3}\sigma_2
\underbrace{\sigma_1^4\sigma_2\cdots\sigma_1^4\sigma_2}_{n-2}\sigma_1^4\sigma_2\sigma_1^3\sigma_2\sigma_1^{n+1}.
\end{eqnarray*}
\end{proof}

Abe and Kishimoto \cite{abe} have independently calculated the Turaev genus for the (3,q)-torus links. We give diagrams in closed braid form that minimize Turaev genus, while they have a different approach.
\begin{proposition}
\label{torgen}
Suppose $q> 0$. The Turaev genus of $T(3,q)$ and $T(3,-q)$ is $\lfloor q/3 \rfloor$.
\end{proposition}
\begin{proof}
Let $D$ be the diagram obtained by taking the closure of the normal form for $(\sigma_1\sigma_2)^q$ given in Lemma \ref{norm}. Thus $D$ is a diagram for $T(3,q)$ and is the closure of a braid in the form
$$\sigma_1^{a_1}\sigma_2^{b_1}\cdots\sigma_1^{a_s}\sigma_2^{b_s}\sigma_1^{a_{s+1}},$$
where $s=\lfloor q/3 \rfloor +1$, both $a_i>0$ and $b_i>0$ for all $1\leq i \leq s$, and $a_{s+1}\geq 0$.
Let $D^\prime$ be the diagram obtained by taking the closure of the braid $(\sigma_1\sigma_2)^{s}$. Corollary \ref{braidreduc} implies that $g(\Sigma_D)=g(\Sigma_{D^\prime})$.
Since $c(D^\prime) = 2\lfloor q/3\rfloor +2, s_0(D^\prime)=3$ and $s_1(D^\prime)=1$, it follows that $g(\Sigma_{D^\prime})=\lfloor q/3 \rfloor$. 
Proposition \ref{bound} and Theorem \ref{turn} imply that the Turaev genus of $T(3,q)$ is greater than or equal to $\lfloor q/3 \rfloor$. Therefore, $g_T(T(3,q)) = \lfloor q/3 \rfloor$. The genera of the Turaev surfaces for a diagram and its mirror are equal, and hence $g_T(T(3,-q))=\lfloor q/3 \rfloor$.
\end{proof}
The next corollary follows directly from Theorem \ref{turn}, Proposition \ref{torgen},  and Proposition \ref{bound}.
\begin{corollary}
\label{zcoeff}
Suppose $n\geq1$.
\begin{enumerate}
\item The group $Kh(T(3,3n))$ is $[4n-3,6n-1]$-thick and the group $Kh(T(3,-3n))$ is\\ $[-6n+1,-4n+3]$-thick. Therefore 
$w_{Kh}(T(3,3n)) = w_{Kh}(T(3,-3n))=n + 2.$
\item The group $Kh(T(3,3n+1))$ is $[4n-1,6n+1]$-thick and the group $Kh(T(3,-3n-1)$ is $[-6n-1,-4n+1]$-thick. Therefore 
$w_{Kh}(T(3,3n+1)) =w_{Kh}(T(3,-3n-1)= n + 2.$
\item The group $Kh(T(3,3n+2))$ is $[4n+1,6n+3]$-thick and the group $Kh(T(3,-3n-2))$ is $[-6n-3,-4n-1]$-thick. Therefore
$w_{Kh}(T(3,3n+2)) =w_{Kh}(T(3,-3n-2))= n + 2.$
\end{enumerate}
\end{corollary}

\subsection{Khovanov width of 3-braids}

In this subsection, we determine the Khovanov width of closed 3-braids based upon Murasugi's classification of closed 3-braids up to conjugation. In \cite{mur}, Murasugi proves the following:
\begin{theorem}[Murasugi]
\label{three}
Let $w\in B_3$ be a braid on three strands, and let $h=(\sigma_1\sigma_2)^3$ be a full twist. Let $n\in\mathbb{Z}$. Then $w$ is conjugate to exactly one of the following:
\begin{enumerate}
\item $ h^n\sigma_1^{p_1}\sigma_2^{-q_1}\cdots\sigma_1^{p_s}\sigma_2^{-q_s}$ where $p_i,q_i$ and $s$ are positive integers.
\item $h^n\sigma_2^m$ where $m\in\mathbb{Z}$.
\item $h^n\sigma_1^m\sigma_2^{-1},$ where $m\in\{-1,-2,-3\}$.
\end{enumerate}
\end{theorem}

Let $L$ be a closed 3-braid. Theorem \ref{three} says, in effect, that $L$ is the closure of a braid of the form $h^n A$. For $n\neq 0$, we say that $L$ has {\it cancellation} if the braid word for $A$ contains a $\sigma_i^\varepsilon$ for $i=1,2$ where $\text{sign}(\varepsilon)\neq\text{sign}(n)$. Besides two infinite family of braids, we prove that $w_{Kh}(L)=|n|+2$ if there is no cancellation and $w_{Kh}(L)=|n|+1$ if there is cancellation.

The following several propositions establish the support of $Kh(L)$. The proofs require the computation of Khovanov homology for a few specific links. We represent the rational Khovanov homology as a Poincare polynomial $P(L)$, a Laurent polynomial in the variables $q$ and $t$ such that the coefficient of $q^i t^j$ is the rank of $Kh^{i,j}(L;\bb{Q})$. These computations were taken from KnotInfo \cite{knotinfo}.

\begin{proposition}
\label{case1lem}
Suppose $n>0$ and $k\geq 0$. Let $D$ be the closure of the braid $(\sigma_1\sigma_2)^{3n}\sigma_1^k\sigma_2^{-1}$, and let $D^\prime$ be the closure of $(\sigma_1\sigma_2)^{-3n}\sigma_1\sigma_2^{-k}$. Then $Kh(D)$ is $[4n+k-2,6n+k-2]$-thick and $Kh(D^\prime)$ is $[-6n-k+2,-4n-k+2]$-thick.
\end{proposition}
\begin{proof}
Observe that $(\sigma_1\sigma_2)^{3n}\sigma_1^k\sigma_2^{-1} = (\sigma_1\sigma_2)^{3n-1}\sigma_1^{k+1}$ for $n>0$. Let $D_+$ be the closure of the braid $(\sigma_1\sigma_2)^{3n-1}\sigma_1$. Resolve the crossing given by the last $\sigma_1$ to obtain two link diagrams $D_v$ and $D_h$. Then $D_v$ is a diagram for $T(3,3n-1)$, and $D_h$ is a diagram for the unknot. By Corollary \ref{zcoeff}, $Kh(D_v)$ is $[4n-3,6n-3]$-thick. Since $D_h$ is the unknot, $Kh(D_h)$ is $[-1,1]$-thick. Recall that $e=\text{neg}(D_h) - \text{neg}(D_+)$. The diagram $D_h$ has $4n-1$ negative crossings, while the diagram $D_+$ has no negative crossings. Thus $e=4n-1$.

If $n\neq 2$, then $D_+$ is width-preserving. If $n=2$, then the Poincare polynomial of $D_+=T(3,5)$ is 
$$P(T(3,5))= q^7+q^9+q^{11}t^2+q^{15}t^3+q^{13}t^4+q^{15}t^4+q^{17}t^5+q^{17}t^6+q^{19}t^5+q^{21}t^7.$$
Therefore,  $Kh^{0,9}(D_v)$ is nontrivial. Moreover, $Kh^{i,j}(D_h)=0$ for all $i$ if $j\leq 9-3e-1=-13$. Therefore, for $n>0$, Proposition \ref{band} implies that $Kh(D)$ is  $[4n+k-2,6n+k-2]$-thick. The proof for $D^\prime$ is similar.
\end{proof}

\begin{proposition}
\label{wid1}
Let $D$ be the closure of the braid $(\sigma_1\sigma_2)^{3n}\sigma_1^{a_1}\sigma_2^{-b_1}\cdots\sigma_1^{a_k}\sigma_2^{-b_k}$, where each $a_i,b_i>0$. Let $a = \sum_{i=1}^k a_i$ and $b = \sum_{i=1}^k b_i$. If $n>0$, then $Kh(D)$ is $[4n+a-b-1,6n+a-b-1]$-thick. If $n<0$, then $Kh(D)$ is $[6n+a-b+1,4n+a-b+1]$-thick. Hence, if $n\neq 0$, then $w_{Kh}(D)=|n|+1$.
\end{proposition}
\begin{proof}
Suppose $n>0$. We proceed by induction on $b$. Suppose $b=1$.
Let $D_1$ be the closure of the braid $(\sigma_1\sigma_2)^{3n}\sigma_1^a\sigma_2^{-1}$. Proposition \ref{case1lem} states that $Kh(D_1)$ is supported in the band $[4n+a-2,6n+a-2]$. Since $(\sigma_1\sigma_2)^3$ is in the center of $B_3$, it follows that $D_1$ represents the same link as $D_1^\prime$, the closure of $(\sigma_1\sigma_2)^{3n}\sigma_1^{a_1}\sigma_2^{-1}\sigma_1^{a-a_1}$. 

If $D_b$ is the closure of the braid $(\sigma_1\sigma_2)^{3n}\sigma_1^{p_1}\sigma_2^{-q_1}\cdots\sigma_1^{p_j}\sigma_2^{-q_j}$ where each $p_i,q_i>0$, $\sum_{i=1}^{j} p_i = a$ and $\sum_{i=1}^j q_i =b$, then by way of induction, suppose $Kh(D_b)$ is $[4n+a-b-1,6n+a-b-1]$-thick.
Let $D_{b+1}$ be the closure of the braid $(\sigma_1\sigma_2)^{3n}\sigma_1^{p^\prime_1}\sigma_2^{-q^\prime_1}\cdots\sigma_1^{p^\prime_l}\sigma_2^{-q^\prime_l}$, where $p^\prime_i,q^\prime_i>0$, $\sum_{i=1}^l p^\prime_i = a$, and $\sum_{i=1}^l q^\prime_i=b+1$. Resolve $D_{b+1}$ at the crossing corresponding to the last $\sigma_2^{-1}$ to obtain diagrams $D_v$ and $D_h$. By the inductive hypothesis, $Kh(D_v)$ is $[4n+a-b-1,6n+a-b-1]$-thick. Let $m$ be the number of negative crossings in the alternating part of $D_h$. The alternating part of $D_h$ has $a+b$ crossings. 

Also, $D_h$ is a non-alternating diagram for an alternating link $L$. Hence, Theorem \ref{qalt} implies that $Kh(L)$ is $[-\sigma(L)-1,-\sigma(L)+1]$-thick. One can calculate the signature of an alternating link from any alternating diagram by a result of Gordon and Litherland \cite{gor}. Color the regions of the alternating diagram in a checkerboard fashion so that near each crossing it looks like Figure \ref{sigcolor}.
\begin{figure}[h]
\includegraphics{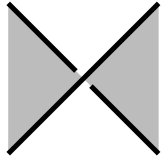}
\caption{Color the alternating diagram in a checkerboard fashion such that a neighborhood of each crossing appears as above.}
\label{sigcolor}
\end{figure}

\noindent Then the signature is given by
$$\sigma(L)=\#(\text{black regions})-\#(\text{positive crossings})-1.$$

There is another diagram representing $L$ that has $b+2$ black regions and $a+b-m$ positive crossings (see Figure \ref{proofex}). Therefore, $\sigma(L)=m-a+3$, and hence $Kh(D^\prime_h)$ is $[a-m-4,a-m-2]$-thick. Since there are $4n$ negative crossing in the full twist part of $D_h$ and $m$ negative crossings in the alternating part of $D_h$, it follows that $\text{neg}(D_h)=4n+m$. Let $D_+$ be the closure of the braid  $(\sigma_1\sigma_2)^{3n}\sigma_1^{p^\prime_1}\sigma_2^{-d^\prime_1}\cdots\sigma_1^{p^\prime_l}\sigma_2^{-d^\prime_l+1}\sigma_2$. Then $\text{neg}(D_+)=b$, and thus $e=\text{neg}(D_h)-\text{neg}(D_+)=4n+m-b$. For $n>0$,
\begin{eqnarray*}
4n+a-b-1 & \neq & (a-m-4) + (4n + m -b) -1\text{ and}\\
6n+a-b-1 & \neq & (a-m-2) + (4n+m-b) -1.
\end{eqnarray*}
Therefore, Theorem \ref{glue} implies that $Kh(D_{b+1})$ is $[4n+a-b-2,6n+a-b-2]$-thick. The proof for $n<0$ is similar.
\end{proof}
\begin{figure}[h]
\includegraphics[scale=.3]{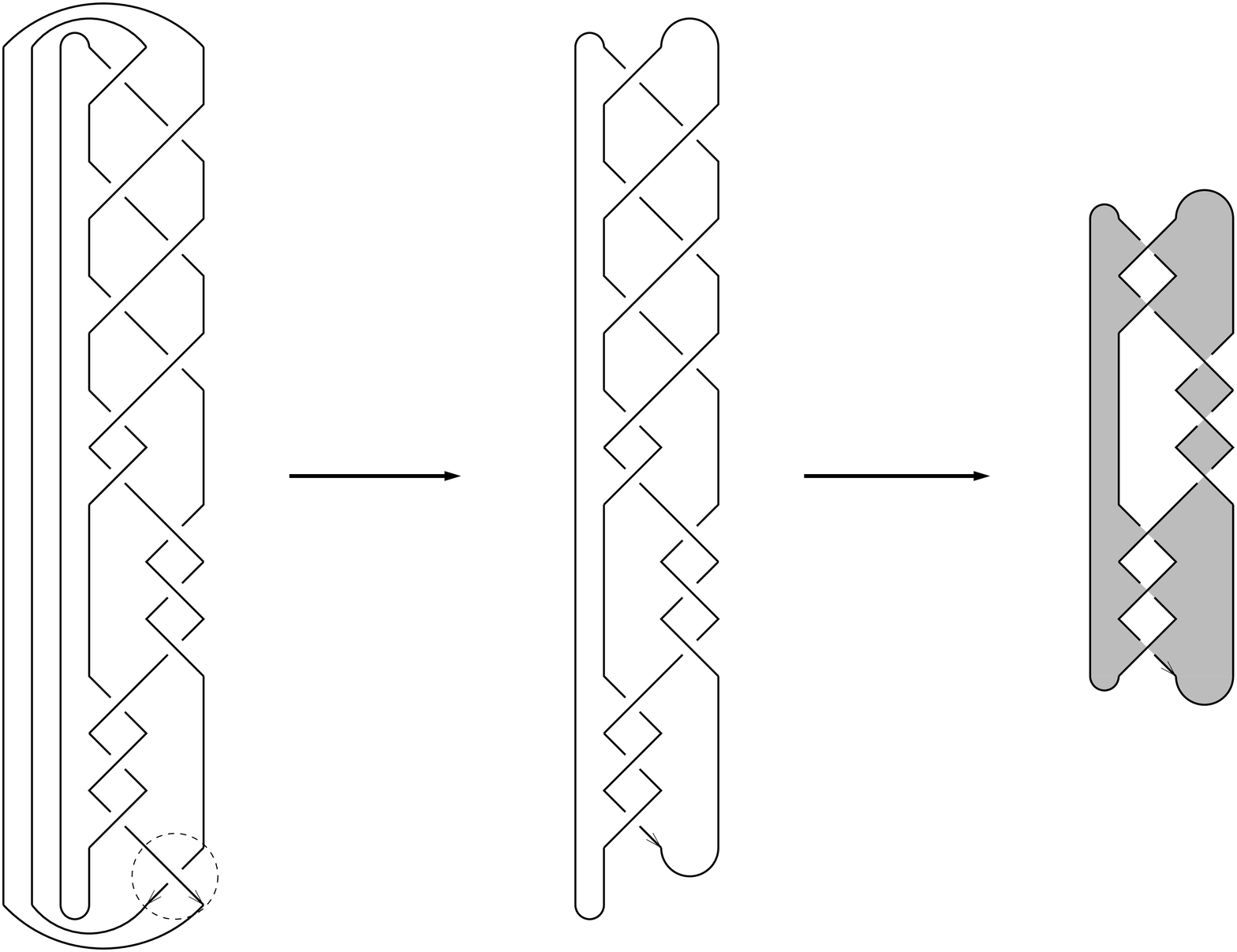}
\caption{The closure of the braid $(\sigma_1\sigma_2)^3\sigma_1^2\sigma_2^{-3}\sigma_1^3\sigma_2^{-1}$ together with its resolution and an alternating diagram of its resolution. There are $5$ black regions and $2$ negative crossings in the alternating diagram.}
\label{proofex}
\end{figure}

\begin{proposition}
\label{wid2}
Let $D$ be the closure of the braid $(\sigma_1\sigma_2)^{3n}\sigma_2^m$.
\begin{enumerate}
\item If $n>0$ and $m\geq 0$, then $Kh(D)$ is $[4n+m-3,6n+m-1]$-thick and $w_{Kh}(D)=n+2$.
\item If $n<0$ and $m\leq 0$, then $Kh(D)$ is $[6n+m+1,4n+m+3]$-thick and $w_{Kh}(D)=-n+2$.
\item If $n=1$ and $m<-3$, then $Kh(D)$ is $[m+3,m+7]$-thick and $w_{Kh}(D)=3.$
\item If $n=-1$ and $m>3$, then $Kh(D)$ is $[m-7,m-3]$-thick and $w_{Kh}(D)=3.$
\item If both $n=1$ and $-3\leq m<0$ or both $n>1$ and $m<0$, then $Kh(D)$ is $[4n+m-1,\newline 6n+m-1]$-thick and $w_{Kh}(D) = n+1$.
\item If both $n=-1$ and $0<m\leq 3$ or both $n<-1$ and $m>0$, then $Kh(D)$ is $[6n+m+1,\newline 4n+m+1]$-thick and $w_{Kh}(D)=-n+1$.
\end{enumerate}
\end{proposition}
\begin{proof}
We prove statements (1), (3), and (5). Statements (2), (4), and (6) are proved similarly.\\
{\bf (1)}. Suppose $n>0$ and $m\geq 0$. Let $D_+$ be the closure of the braid $(\sigma_1\sigma_2)^{3n}\sigma_2$. Resolve $D_+$ at the crossing corresponding to the last $\sigma_2$ to obtain diagrams $D_v$ and $D_h$. Then $D_v$ is a diagram for $T(3,3n)$. By Corollary \ref{zcoeff}, $Kh(D_v)$ is $[4n-3,6n-1]$-thick. Also, $D_h$ is the two component unlink, and hence $Kh(D_h)$ is $[-2,2]$-thick. The diagram $D_h$ has $4n$ negative crossings, and the diagram $D_+$ has no negative crossings. Thus $e= 4n$. 

Observe that $4n-3\neq -2 + e +1$ and $6n-1 =  2 + e + 1$ when $n=2$. If $n=2$, then $D_v$ is $T(3,6)$, and 
$$P(T(3,6))= q^9+q^{11}+q^{13}t^2+q^{17}t^3+q^{15}t^4+q^{17}t^4+q^{19}t^5+q^{19}t^6+q^{21}t^7+q^{21}t^8+q^{23}t^7+3q^{23}t^8+2q^{25}t^8.$$
Therefore $Kh^{0,11}(D_v)$ is nontrivial. Also, $Kh^{i,j}(D_h)=0$ for all $i$ if $j\leq 11-3e-1 = -14$. Hence, Theorem \ref{band} implies that $Kh(D)$ is $[4n+m-3,6n+m-1]$-thick.

{\bf (3)}. Suppose $n=1$ and $m<-3$. Let $D_-$ be the closure of $(\sigma_1\sigma_2)^{3}\sigma_2^{-5}$. Resolve $D_-$ at the crossing corresponding to the last $\sigma_2^{-1}$ to obtain diagrams $D_v$ and $D_h$. The diagram $D_h$ is a diagram for the two component unlink, and hence $Kh(D_h)$ is $[-2,2]$-thick. The diagram $D_v$ is a diagram for the link $L(6,n,1)$ in Thistlethwaite's link table (see Figure \ref{L6n1}).
\begin{figure}[h]
\includegraphics[scale=.3]{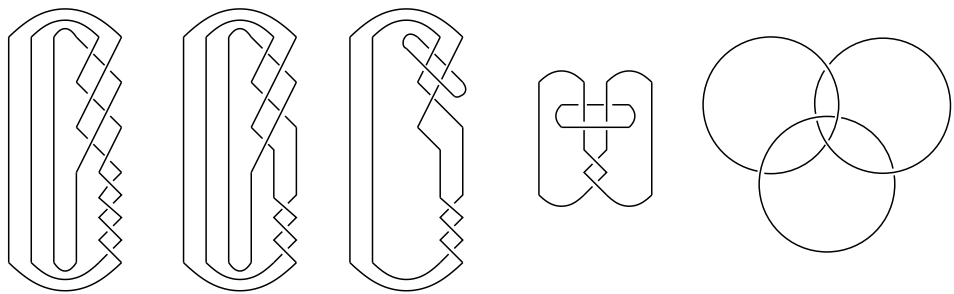}
\caption{A transformation of the closure of $(\sigma_1\sigma_2)^3\sigma_2^{-4}$ into L(6,n,1).}
\label{L6n1}
\end{figure}
The Poincare polynomial for L(6,n,1) is given by
$$P(L(6,n,1))=2q^{-1}+3q+q^3+qt+q^5t^2+q^7t^4+q^9t^4.$$
Therefore $Kh(D_v)$ is $[-1,3]$-thick. The diagram $D_h$ has $4$ negative crossings while the diagram $D_+$ also has $4$ negative crossings. Thus $e=0$. Since $-1 \neq -2 + e - 1$ and $3 \neq 2 + e -1$, Theorem \ref{band} implies that $Kh(D)$ is $[m+3,m+7]$-thick.

{\bf (5)}. If $n=1$ and $-3\leq m<0$, then Baldwin \cite{bald} has shown that $D$ is quasi-alternating. Therefore, Theorem \ref{qalt} implies that $Kh(D)$ is $[-\sigma(L)-1,-\sigma(L)+1]$-thick, where $L$ is the link type of $D$. A straightforward calculation of signature gives the desired result.

Suppose $n>1$ and $m<0$. Observe that $(\sigma_1\sigma_2)^{3n}\sigma_2^{-1} =(\sigma_1\sigma_2)^{3n-1}\sigma_1$. Let $D_+$ be the closure of the braid $(\sigma_1\sigma_2)^{3n-1}\sigma_1$. Resolve $D_+$ at the crossing corresponding to the last $\sigma_1$ to obtain diagrams $D_v$ and $D_h$. Since $D_v$ is a diagram for $T(3,3n-1)$, it follows that $Kh(D_v)$ is $[4n-3,6n-3]$-thick. Since $D_h$ is a diagram for the unknot, it follows that $Kh(D_h)$ is $[-1,1]$-thick. The diagram $D_h$ has $4n-1$ negative crossings, and $D_v$ has no negative crossings. Thus $e=4n-1$. 

If $n=2$, then $6n-3 = 1 + e + 1$, and the long exact sequence of Theorem \ref{les} looks like
$$0\to Kh^{0,10}(D_+)\to Kh^{0,9}(D_v)\to Kh^{-7,-13}(D_h)\to\cdots.$$
Since $Kh^{-7,-13}(D_h)=0$ and $Kh^{0,9}(D_v)$ is nontrivial, it follows that $Kh^{0,10}(D_+)$ is nontrivial. Since $4n-3 \neq -1 + e + 1$ and $6n-3 \neq 1 + e + 1$ for $n>2$, Corollary \ref{glue} implies that $Kh(D_+)$ is $[4n-2,6n-2]$-thick.

Let $D_-$ be the closure of $(\sigma_1\sigma_2)^{3n-1}\sigma_1\sigma_2^{-1}$. Resolve $D_-$ at the crossing given by the last $\sigma_2^{-1}$ to obtain diagrams $D_v$ and $D_h$. The diagram $D_v$ is the closure of the braid $(\sigma_1\sigma_2)^{3n-1}\sigma_1$, and hence $Kh(D_v)$ is $[4n-2,6n-2]$-thick. The link $D_h$ is a diagram for the two component unlink, and thus $Kh(D_h)$ is $[-2,2]$-thick. The diagram $D_h$ has $4n-1$ negative crossings, and $D_+$ has no negative crossings. Thus $e=4n-1$.

For $n>1$, we have $4n-2 \neq -2 + e -1$ and $6n-2 \neq 2 + e -1$. Therefore, Theorem \ref{band} implies that $Kh(D)$ is $[4n+m-1,6n+m-1]$-thick. 
\end{proof}

\begin{proposition}
\label{wid3}
Let $D$ be the closure of the braid $(\sigma_1\sigma_2)^{3n}\sigma_1^m\sigma_2^{-1}$, where $m\in\{-1,-2,-3\}$.
\begin{enumerate}
\item If $n>0$, then $Kh(D)$ is $[4n+m-2,6n+m-2]$-thick, and $w_{Kh}(D)=n+1$.
\item If $n<0$, then $Kh(D)$ is $[6n+m,4n+m+2]$-thick, and $w_{Kh}(D)=-n+2$.
\end{enumerate}
\end{proposition}
\begin{proof}
{\bf (1)}. Suppose $n>0$. If $m= -1$, then $D$ is a diagram for $T(3,3n-1)$, and the result follows. 

Let $m=-2$. Then, up to conjugation in $B_3$, we have
\begin{eqnarray*}
(\sigma_1\sigma_2)^{3n}\sigma_1^{-2}\sigma_2^{-1} & = & (\sigma_1\sigma_2)^{3n-1}\sigma_1^{-1}\\
&  = & (\sigma_2\sigma_1)^{3n-2} \sigma_2\\
& = & (\sigma_1\sigma_2)^{3n-2}\sigma_1.
\end{eqnarray*}
If $D^\prime$ is the closure of the braid $(\sigma_1\sigma_2)^{3n-2}\sigma_1$, then $D$ and $D^\prime$ represent the same link. Resolve $D^\prime$ at the crossing corresponding to the final $\sigma_1$ to obtain diagrams $D_v$ and $D_h$. Then $D_v$ is a diagram for $T(3,3n-2)$, and $D_h$ is a diagram for the unknot. Hence $Kh(D_v)$ is $[4n-5,6n-5]$-thick, and $Kh(D_h)$ is $[-1,1]$-thick. The diagram $D_h$ has $4n-3$ negative crossings, and the diagram $D^\prime$ has none. Thus $e=4n-3$. 

Observe that $4n-5\neq -1 + e + 1$, and $6n-5 = 1 + e +1$ when $n=2$. If $n=2$, the long exact sequence of Theorem \ref{les} looks like
$$0\to Kh^{0,8}(D^\prime)\to Kh^{0,7}(D_v)\to Kh^{-5,-9}(D_h)\to\cdots$$
Since $Kh^{-5,-9}(D_h)=0$ and $Kh^{0,7}(L_v)$ is nontrivial, it follows that $Kh^{0,8}(L)$ is nontrivial. Hence Theorem \ref{glue} implies that $Kh(D^\prime)$ is $[4n-4,6n-4]$-thick.

Let $m=-3$. Then, up to conjugation in $B_3$, we have
\begin{eqnarray*}
(\sigma_1\sigma_2)^{3n}\sigma_1^{-3}\sigma_2^{-1} & = &(\sigma_1\sigma_2)^{3n-1}\sigma_1^{-2}\\
& = & (\sigma_2\sigma_1)^{3n-3}\sigma_2\sigma_1\sigma_2\sigma_1^{-1}\\
& = & (\sigma_1\sigma_2)^{3n-2}.
\end{eqnarray*}
Hence $D$ is a diagram for $T(3,3n-2)$, and the result follows.

{\bf (2)}. Let $n<0$. If $m=-1$, then $D$ is a diagram for $T(3,3n-1)$, and the result follows.

Let $m=-2$. Then $D$ is the closure of $(\sigma_1\sigma_2)^{3n-1}\sigma_1^{-1}$. Resolve $D$ at the crossing corresponding to the last $\sigma_1^{-1}$ to obtain diagrams $D_v$ and $D_h$. Then $D_v$ is a diagram for $T(3,3n-1)$, and hence $Kh(D_v)$ is  $[6n-1,4n+1]$-thick. Also, $D_h$ is a diagram for the unknot, and hence $Kh(D_h)$ is $[-1,1]$-thick. The diagram $D_h$ has $-2n-1$ negative crossings, and the diagram $D_+$ has $-6n+2$ negative crossings. Thus $e=4n-1$.

Observe that $4n+1 \neq 1 + e -1$, and $6n-1 = -1 +e -1 $ if $n=-1$. If $n=-1$, the long exact sequence of Theorem \ref{les} looks like 
$$\cdots\to Kh^{5,9}(D_h)\to Kh^{0,-7}(D_v) \to Kh^{0,-8}(D)\to 0.$$
Since $Kh^{5,9}(D_h)=0$ and $Kh^{0,-7}(D_v)$ is nontrivial, it follows that $Kh^{0,-8}(D)$ is nontrivial. Thus $Kh(D)$ is $[6n-2,4n]$-thick.

Let $m=-3$. Then, up to conjugacy in $B_3$, we have 
\begin{eqnarray*}
(\sigma_1\sigma_2)^{3n}\sigma_1^{-3}\sigma_2^{-1} & = & (\sigma_1\sigma_2)^{3n}\sigma_1^{-1}\sigma_2^{-1}\sigma_1^{-2}\\
& = & (\sigma_1\sigma_2)^{3n-2}.
\end{eqnarray*}
In this case $D$ is a diagram for $T(3,3n-2)$, and the result follows.
\end{proof}

We collect the results of Propositions \ref{wid1}, \ref{wid2} and \ref{wid3} into one theorem giving the Khovanov width of closed 3-braids.
\begin{theorem}
\label{3width}
Let $L$ be a closed 3-braid of the form $h^nA$, as in Theorem \ref{three}, where $h=(\sigma_1\sigma_2)^3$ and $n\neq 0$. Then
\begin{equation*}
w_{Kh}(L) = 
\begin{cases}
|n|+2 & \text{if}~L~\text{has no cancellation or if}~L~ \text{is the closure of}~h^{\pm 1}\sigma_2^{\mp m}~\text{where}~m>3,\\
|n|+1 & \text{otherwise.}
\end{cases}
\end{equation*}
\end{theorem}
\begin{remark}
If $n=0$, then $L$ is a (possibly split) alternating link, and thus $w_{Kh}(L)$ can be deduced from Theorem \ref{qalt} and Proposition \ref{kholink}.
\end{remark}

In \cite{bald}, Baldwin classifies quasi-alternating closed 3-braids.
\begin{proposition}
\label{qaltbraids}
Let $L$ be a closed 3-braid and let $h=(\sigma_1\sigma_2)^3$.
\begin{itemize}
\item If $L$ is the closure of the braid $h^n\sigma_1^{a_1}\sigma_2^{-b_2}\cdots\sigma_1^{a_k}\sigma_2^{-b_k}$, where each $a_i,b_i>0$, then $L$ is quasi-alternating if and only if $n\in\{-1,0,1\}$.
\item If $L$ is the closure of the braid $h^n\sigma_2^m$, then $L$ is quasi-alternating if and only if either $n=1$ and $m\in\{-1,-2,-3\}$ or $n=-1$ and $m\in\{1,2,3\}$.
\item If $L$ is the closure of the braid $h^n\sigma_1^m\sigma_2^{-1}$, where $m\in\{-1,-2,-3\}$, then $L$ is quasi-alternating if and only if $n\in\{0,1\}$.
\end{itemize}
\end{proposition}

Using the spectral sequence from reduced Khovanov homology of a link to the Heegaard Floer homology of the branched double cover of that link, Baldwin \cite{bald} shows the following corollary. This corollary is also a consequence of Theorem \ref{3width} and Proposition \ref{qaltbraids}.
\begin{corollary}[Baldwin]
\label{thin}
Let $L$ be a closed 3-braid. Then $L$ is quasi-alternating if and only if $w_{\wt{Kh}}(L)=1$.
\end{corollary}
\begin{remark}
\label{shum}
Shumakovitch has shown that the $9_{46}$ and $10_{140}$ knots (both closed $4$-braids) have reduced Khovanov width 1, but they are not quasi-alternating. One can use either of these knots to generate infinite families of counterexamples to Corollary \ref{thin} for braids with index greater than 3.
\end{remark}

\subsection{Turaev genus of closed 3-braids}

Combining Lemma \ref{norm} with Corollary \ref{braidreduc} gives a useful tool to compute the Turaev genus of closed 3-braids. By using the lower bound given by Proposition \ref{bound}, the Turaev genus  of closed 3-braids can be calculated up to a maximum additive error of at most $1$. 

\begin{proposition}
\label{tur1}
Let $L$ be the link type closure of $(\sigma_1\sigma_2)^{3n}\sigma_1^{a_1}\sigma_2^{-b_1}\cdots\sigma_1^{a_k}\sigma_2^{-b_k}$, where each $a_i,b_i>0$ and $n\neq 0$. Then $|n|-1\leq g_T(L)\leq |n|$.
\end{proposition}
\begin{proof}
Suppose $n>0$. We have
$$(\sigma_1\sigma_2)^{3n}\sigma_1^{a_1}\sigma_2^{-b_1}\cdots\sigma_1^{a_k}\sigma_2^{-b_k} =  (\sigma_1\sigma_2)^{3n-1}\sigma_1^{a_1+1}\sigma_2^{-b_1}\cdots\sigma_1^{a_k}\sigma_2^{-b_k+1}.$$
If $b_k>1$, let $D$ be the closure of the braid $(\sigma_1\sigma_2)^n(\sigma_1\sigma_2^{-1})^k$ and if $b_k=1$, let $D$ be the closure of the braid $(\sigma_1\sigma_2)^n(\sigma_1\sigma_2^{-1})^{k-1}$. By applying the normal form of Lemma \ref{norm} to $(\sigma_1\sigma_2)^{3n-1}$ and then using Corollary \ref{braidreduc}, it follows that $g_T(L)\leq g_T(D)$. A straightforward calculation shows that $g_T(D)=n$. Since $w_{Kh}(L)=n+1$ and $w_{Kh}(L)-2\leq g_T(L)$, we have $n-1\leq g_T(L)$. The case where $n<0$ is similar.
\end{proof}
\begin{proposition}
\label{tur2}
Let $L$ be the link type of the closure of $(\sigma_1\sigma_2)^{3n}\sigma_2^m$, where $n\neq 0$. 
\begin{enumerate}
\item If $L$ has no cancellation, then $g_T(L)=|n|$.
\item If $L$ has cancellation and $|n|>1$, then $|n|-1\leq g_T(L)\leq |n|$.
\item If either both $n=1$ and $-3\leq m <0$ or both $n=-1$ and $0<m\leq 3$, then $g_T(L)=0$.
\item If either both $n=1$ and $m<-3$ or both $n=-1$ and $m>3$. Then $g_T(L)=1$.
\end{enumerate}
\end{proposition}
\begin{proof}
{\bf (1).} If $L$ has no cancellation, then either both $n>0$ and $m\geq 0$ or $n<0$ and $m\leq 0$. Corollary \ref{braidreduc} implies that $g_T(L)\leq g_T(T(3,3n))=|n|$. Since $w_{Kh}(L)=|n|+2$, it follows that $g_T(L)=|n|$.

{\bf (2).} Suppose that $L$ has cancellation and $n>1$. Then $m<0$ and
$$(\sigma_1\sigma_2)^{3n}\sigma_2^{m}=(\sigma_1\sigma_2)^{3n-1}\sigma_1\sigma_2^{(m+1)}.$$
If $m<-1$, let $D$ be the closure of $(\sigma_1\sigma_2)^n\sigma_1\sigma_2^{-1}$, and if $m=-1$, let $D$ be the closure $(\sigma_1\sigma_2)^n$. Lemma \ref{norm} and Corollary \ref{braidreduc} imply that $g_T(L)\leq g(\Sigma_D)$. A straightforward calculation shows that $g(\Sigma_D)=n$. Since $w_{Kh}(L)=n+1$, it follows that $n-1\leq g_T(L)$. The case where $n<-1$ and $m>0$ is similar.

{\bf (3).} Suppose $n=1$ and $-3\leq m<0$. As noted in Baldwin's paper \cite{bald}, we have
$$(\sigma_1\sigma_2)^3\sigma_2^m=\sigma_1\sigma_2^2\sigma_1\sigma_2^2\sigma_2^m.$$
By canceling the $\sigma_2^m$ with the final $\sigma_2^2$, one obtains a diagram for $L$ with $5$ crossings or less. Therefore $L$ is alternating and $g_T(L)=0$. The case where $n=-1$ and $0<m\leq 3$ is similar.

{\bf (4).} Suppose $n=1$ and $m<-3$. Then $L$ can be represented by the closure of $\sigma_1\sigma_2^2\sigma_1\sigma_2^{m+2}$. Let $D$ be the closure of the braid $\sigma_1\sigma_2\sigma_1\sigma_2^{-1}$. By Corollary \ref{braidreduc}, we have $g_T(L)\leq g(\Sigma_D)$, and a straightforward calculation shows that $g(\Sigma_D)=1$. Since $w_{Kh}(L)=3$, it follows that $g_T(L)=1$. The case where $n=-1$ and $m>3$ is similar.
\end{proof}
\begin{proposition}
\label{tur3}
Let $L$ be the link type of the closure of $(\sigma_1\sigma_2)^{3n}\sigma_1^m\sigma_2^{-1}$ where $m\in\{-1,-2,-3\}$. If $n>0$, then $g_T(L)=n-1$ and if $n<0$, then $g_T(L)=|n|$.
\end{proposition}
\begin{proof}
Let $n>0$. Using the forms in the proof of Proposition \ref{wid3} and the reductions of Lemma \ref{norm} and Corollary \ref{braidreduc}, one sees that $g_T(L)\leq g(\Sigma_D)$ where $D$ is the closure of $(\sigma_1\sigma_2)^{n}$. A straightforward calculation shows that $g(\Sigma_D)=n-1$. Since $w_{Kh}(L)=n+1$, it follows that $g_T(L)=n-1$.

Let $n<0$. Using the forms in the proof of Proposition \ref{wid3} and the reductions of Lemma \ref{norm} and Corollary \ref{braidreduc}, one sees that $g_T(L)\leq g(\Sigma_{D^\prime})$ where $D^\prime$ is the closure of $(\sigma_1\sigma_2)^{n+1}$. A straightforward calculation shows that $g(\Sigma_{D^\prime})=|n|$. Since $w_{Kh}(L)=|n|+2$, it follows that $g_T(L)=|n|$. 
\end{proof}

The previous  results of this section are summarized in the following corollary.
\begin{corollary}
Let $L$ be a closed 3-braid. Then
$$0 \leq g_T(L)-(w_{Kh}(L)-2) \leq 1.$$
\end{corollary}
\begin{remark}
Both the lower bound and upper bound of the above inequality are achieved by closed 3-braids. For example, the links in Proposition \ref{tur3} achieve the lower bound while the links in Proposition \ref{tur2} part (4) achieve the upper bound. There are also closed 3-braids (see Proposition \ref{tur1}) where it is unknown whether the lower bound or upper bound is achieved.
\end{remark}

\section{Applications to odd Khovanov homology}
\label{oddsec}

In \cite {odd}, Ozsv\'ath, Rasmussen and Szab\'o introduced odd Khovanov homology, a knot homology that is closely related to Khovanov homology. Odd Khovanov homology, denoted $Kh_{\text{odd}}(L)$, is a bigraded $\bb{Z}$-module whose graded Euler characteristic is the unnormalized Jones polynomial. 

\subsection{A spanning tree model for odd Khovanov homology} Champanerkar and Kofman \cite{span} and independently Wehrli \cite{wehrli} developed a spanning tree model for Khovanov homology. In this subsection, we show that the similarities between Khovanov homology and odd Khovanov homology imply that odd Khovanov homology also has a spanning tree model. 

Let $D$ be a link diagram and let $\mathcal{X}$ be the set of crossings of $D$. Suppose $(C(D),\partial)$ is the hypercube of resolutions complex from \cite{kho} and \cite{barnatan} that generates Khovanov homology, and suppose $(C_{\text{odd}}(D),\partial_{\text{odd}})$ is the hypercube of resolutions complex from \cite{odd} that generates odd Khovanov homology. A vertex in the hypercube is a function $I:\mathcal{X}\to\{0,1\}$. For each vertex $I$, one obtains a one-manifold $D_I$ be smoothing each crossing of $D$ according to $I$. Both chain complexes are constructed by associating certain $\bb{Z}$-modules to each of the one-manifolds $D_I$.

Number the crossings of $D$ from $1$ to $|\mathcal{X}|$ arbitrarily. One can obtain the vertices of the hypercube as the leaves of a binary tree. The root of this tree is the diagram $D$. The children of a vertex $v$ at level $i$ are obtained by smoothing the $i$th crossing of $v$ into either a $0$-resolution or a $1$-resolution. See Figure \ref{trefskein}.
\begin{figure}[h]
\includegraphics[scale=.2]{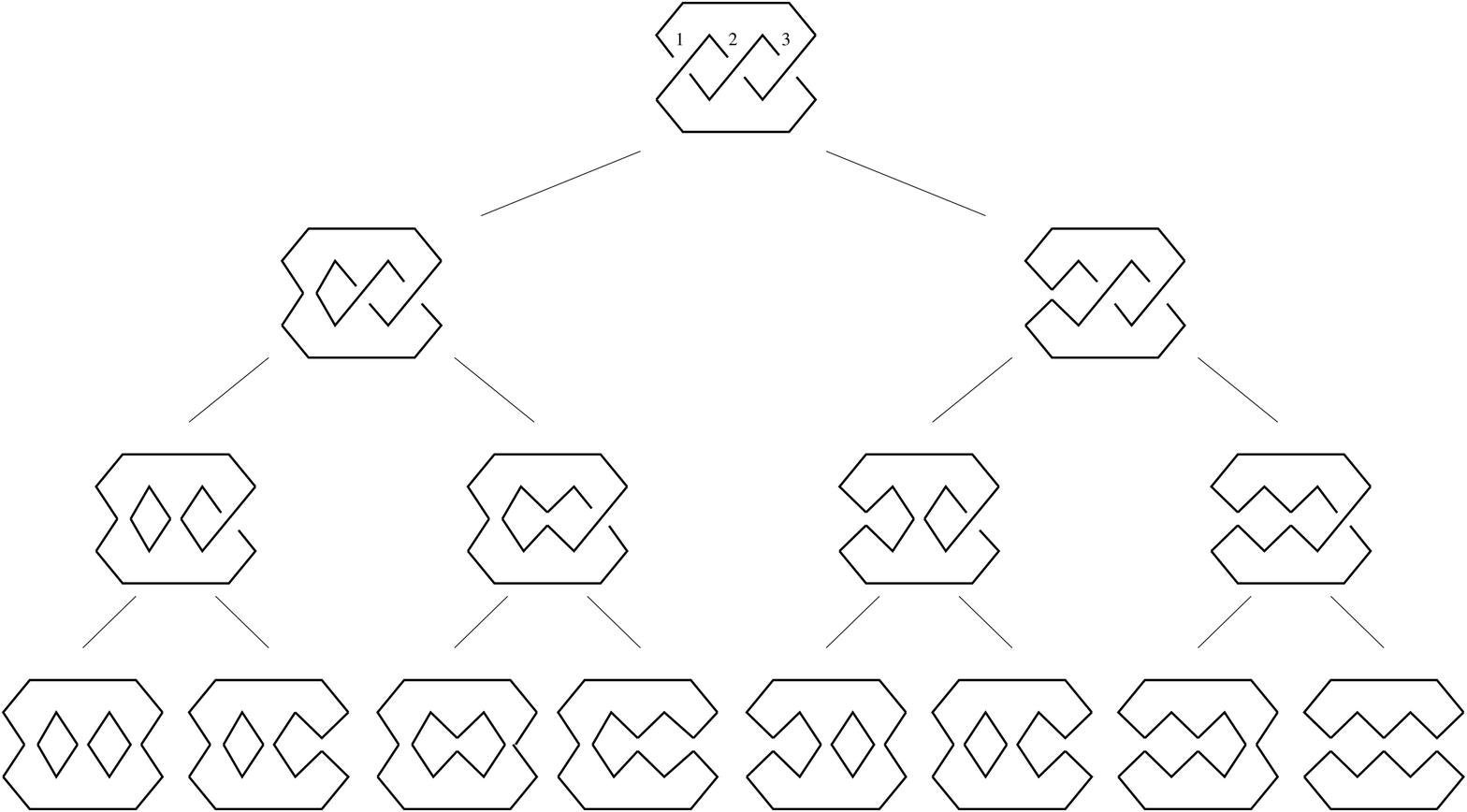}
\caption{Binary tree whose leaves are the vertices of the hypercube.}
\label{trefskein}
\end{figure}

Modify the binary tree as follows. If either of the children of a vertex $v$ is disconnected, then the vertex $v$ becomes a leaf and all its descendants are deleted. See Figure \ref{treftrees}. The leaves of the modified binary trees are {\it twisted unknots}, i.e. they are unknots that can be trivialized using only Reidemeister I moves. Also, the leaves are in one-to-one correspondence with the spanning trees of either checkerboard graph associated to $D$. The details of this correspondence are described in Champarnerkar-Kofman \cite{span} and Wehrli \cite{wehrli}. Denote the set of spanning trees by $\mathcal{T}(D)$, and the diagram associated to a tree $T\in\mathcal{T}(D)$ by $D_T$.
\begin{figure}[h]
\includegraphics[scale=.2]{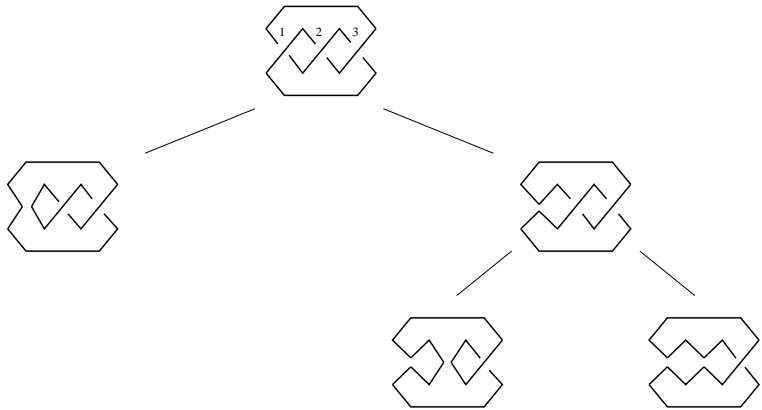}
\caption{The modified binary tree whose leaves are in correspondence to spanning trees of the checkerboard graph of $D$.}
\label{treftrees}
\end{figure}

Let $U$ denote diagram of the unknot with no crossings. The Khovanov complex of the disjoint union of $k$ copies of $U$ is given by
$$C(U^k) = \mathcal{A}^{\otimes k},$$
where $\mathcal{A}$ is the bigraded module defined by $\mathcal{A}^{0,-1}=\mathcal{A}^{0,1}=\bb{Z}$ and $\mathcal{A}^{i,j}=0$ for $(i,j)\neq(0,\pm 1)$. For any bigraded object $M$, define grading shifts $[m]$ and $\{n\}$ by $(M[n]\{n\})^{i,j} := M^{i-m.j-n}$.

In \cite{wehrli}, Wehrli gives the following {\it spanning tree model} for Khovanov homology. Champanerkar and Kofman prove an analogous result in \cite{span}.
\begin{proposition}[Wehrli]
\label{wehrspan}
Let $D$ be a connected link diagram. Then there is a decomposition $C(D)=A\oplus B$, where $B$ is contractible and $A$ as a module is given by
$$A = \bigoplus_{T\in\mathcal{T}(D)}\mathcal{A}[f(D,D_T)]\{g(D,D_T)\},$$
for functions $f$ and $g$ depending on $D$ and $D_T$.
\end{proposition}

Let $D$, $D_0$ and $D_1$ be as in Figure \ref{unresol}. The spanning tree model for Khovanov homology is a consequence of
\begin{enumerate}
\item the bigraded $\bb{Z}$-module structure of $C(D)$,
\item the fact that $C(D)$ is isomorphic to the mapping cone of $w: C(D_0)\to C(D_1)$, for some map $w$, and
\item the structure of the complex under Reidemeister I moves, which is specified by
\begin{equation*}  
C(\includegraphics[scale=0.08]{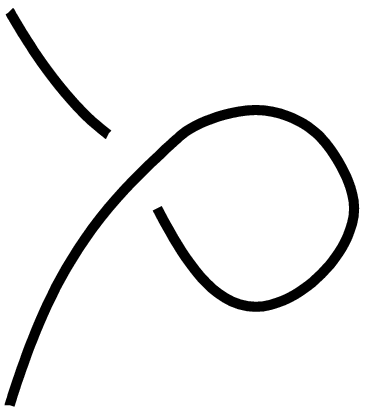}) \cong
C(\includegraphics[scale=0.07]{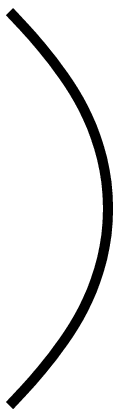})\{-1\} \oplus B_1,\qquad
C(\includegraphics[scale=0.08]{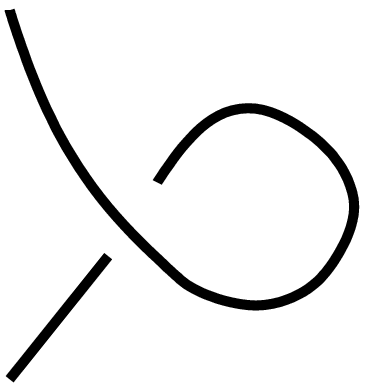}) \cong 
C(\includegraphics[scale=0.07]{notwist.eps})[1]\{2\} \oplus B_2 ,
\end{equation*}
for contractible complexes $B_1$ and $B_2$. 
\end{enumerate}

As bigraded $\bb{Z}$-modules $C(D)$ and $C_{\text{odd}}(D)$ are isomorphic. Furthermore, from the proof of invariance under Reidemeister I moves in \cite{odd}, one can see that (2) and (3) also hold for odd Khovanov homology. Therefore odd Khovanov homology also has a spanning tree model.
\begin{proposition}
\label{oddspan}
Let $D$ be a connected link diagram. Then there is a decomposition $C_{\text{odd}}(D)=A\oplus B$, where $B$ is contractible and $A$ as a module is given by
$$A = \bigoplus_{T\in\mathcal{T}(D)}\mathcal{A}[f(D,D_T)]\{g(D,D_T)\},$$
for functions $f$ and $g$, which are the same as in Proposition \ref{wehrspan}.
\end{proposition}

Proposition \ref{bound} is a consequence of the bigraded $\bb{Z}$-module structure of the spanning tree complex for Khovanov homology. Since odd Khovanov homology has a spanning tree complex with the same bigraded $\bb{Z}$-module structure, there is an analogous Turaev genus bound on the odd Khovanov width of a link $L$, denoted $w_{Kh_{\text{odd}}}(L)$.
\begin{proposition}
\label{oddbound}
Let $L$ be a link. Then 
$$w_{Kh_{\text{odd}}}(L)-2\leq g_T(L).$$
\end{proposition}

\subsection{The odd Khovanov width of closed 3-braids}

There is a close relationship between Khovanov homology and odd Khovanov homology. Ozsv\'ath, Rasmussen, and Szab\'o \cite{odd} have shown that
$$Kh(L;\bb{Z}_2) \cong Kh_{\text{odd}}(L;\bb{Z}_2),$$
and that odd Khovanov homology satisfies long exact sequences identical to the sequences in Theorem \ref{les}.  These similarities, along with the Turaev genus bound given in Proposition \ref{oddbound}, imply the following result.
\begin{theorem}
\label{oddband}
Let $L$ be a closed 3-braid. Then $Kh(L)$ is $[\delta_{\text{min}},\delta_{\text{max}}]$-thick if and only if $Kh_{\text{odd}}(L)$ is $[\delta_{\text{min}},\delta_{\text{max}}]$-thick.
\end{theorem}
\begin{proof}
Let $L'$ be a link that is the base case for one of the inductions in Propositions \ref{case1lem} through \ref{wid3}.
Then $L'$ is either a $(3,q)$ torus link or the link $L(6,n,1)$, and 
\begin{equation}
\label{basewidth}
w_{Kh}(L';\bb{Q}) = w_{Kh_{\text{odd}}}(L') = g_T(L') +2.
\end{equation}

Suppose $Kh^\delta(L';\bb{Q})$ is nontrivial. Then $Kh^\delta(L';\bb{Z}_2)$ is also nontrvial. Since $Kh(L;\bb{Z}_2) \cong Kh_{\text{odd}}(L;\bb{Z}_2)$, it follows that $Kh_{\text{odd}}^\delta(L';\bb{Z}_2)$ is nontrvial. Therefore, $Kh_{\text{odd}}^\delta(L')$ is nontrivial. Then Equation \ref{basewidth} implies 
$$g_T(L')+2 = w_{Kh}(L';\bb{Q}) \leq w_{Kh_{\text{odd}}}(L') \leq g_T(L')+2.$$
Thus $Kh(L')$ is $[\delta_{\text{min}}',\delta_{\text{max}}']$-thick if and only if $Kh'(L')$ is $[\delta_{\text{min}}',\delta_{\text{max}}']$-thick. 

The proofs of Propositions \ref{case1lem} through \ref{wid3} rely only on the Khovanov homology of the base case and the long exact sequences of Theorem \ref{les}. Therefore, Propositions \ref{case1lem} through \ref{wid3} hold for odd Khovanov homology, and this implies the result.
\end{proof}

\begin{corollary}
\label{oddwidth}
Let $L$ be a closed 3-braid. Then
$$w_{Kh}(L) = w_{Kh_{\text{odd}}}(L).$$
\end{corollary}

Note that Corollary \ref{oddwidth} is not true for the closure of $n$-braids where $n>3$. The examples from Remark \ref{shum} have $w_{Kh_{\text{odd}}}(L)>2$. These examples can be used to generate infinite families of examples of closed $n$-braids where odd Khovanov width and Khovanov width are different for $n>3$.


\begin{thebibliography}{99}

\bibitem{abe} T. Abe and K. Kishimoto. {\it The dealternating number and the alternation number of a closed $3$-braid}. arXiv:math.GT/0808.0573.

\bibitem{ap} M. Asaeda and J. Prztycki. {\it Khovanov homology: torsion and thickness}, Advances in topological quantum field theory, Proc. of the NATO Advanced Research Workshop on new techniques in topological quantum field theory. NATA Sci. Ser. II. Math. Phys. Chem. {\bf 179} (2004) 135-166. arXiv:math.GT/0402402.

\bibitem{birman} J. Birman and W. Menasco. {\it Studying links via closed braids III: classifying links which are closed 3-braids.} Pacific Journal of Mathematics. {\bf 161} (1993) 25-115.

\bibitem{bald} J. Baldwin. {\it Heegaard Floer homology and genus one, one boundary component open books}. Journal of Topology. {\bf 1} No. 4 (2008) 962-992. arXiv:math.GT/0804.3624.

\bibitem{barnatan} D. Bar Natan. {\it On Khovanov's categorification of the Jones polynomial.} Algebraic and Geometric Topology {\bf 2} (2002) 337-370. arXiv:math.QA/0201043.

\bibitem{knotinfo} J. C. Cha and C. Livingston. {\it KnotInfo: Table of Knot Invariants}, http://www.indiana.edu/~knotinfo.

\bibitem{span} A. Champanerkar and I. Kofman. {\it Spanning trees and Khovanov homology.} to appear in Proceedings of the American Mathematical Society. arXiv:math.GT/0607510.

\bibitem{cha-kof} A. Champanerkar and I. Kofman. {\it Twisting quasi-alternating links}. arXiv:math.GT/0712.2590.

\bibitem{stoltz} A. Champanerkar, I. Kofman, and N. Stoltzfus. {\it Graphs on Surfaces and Khovanov homology}. Algebraic and Geometric Topology. {\bf 7} (2007) 1531-1540. arXiv:math/0705.3453.

\bibitem{das} O. T. Dasbach, D. Futer, E. Kalfagianni, X. Lin, and N. Stoltzfus. {\it The Jones polynomial and graphs on surfaces}. Journal of Combinatorial Theory, Series B. {\bf 98} Issue 2 (2008) 384-399. arXiv:math.GT/0605571v3.

\bibitem{gar} F. Garside. {\it The braid groups and other groups}. Quart. J. Math. Oxford, {\bf 20} No. 78 (1969), 235-254.

\bibitem{gor} C. Gordon and R  Litherland. {\it On the signature of a link}. Invent. Math. 47 (1978) 53Ð69.

\bibitem{kho} M. Khovanov. {\it A categorification of the Jones polynomial}. Duke Math J. {\bf 101} No. 3 (2000) 359-426. arXiv:math.QA/9908171. 

\bibitem{kho2} M. Khovanov. {\it Patterns in knot cohomology I}. Experiment. Math. {\bf 12} No. 3 (2003), 365-74. arXiv:math.QA/0201306.

\bibitem{lee} E. S. Lee. {\it The support of Khovanov's invariants for alternating knots}. arXiv:math.GT/0201105.

\bibitem{mo} C. Manolescu and P. Ozsv\'{a}th. {\it On the Khovanov and knot Floer homologies of quasi-alternating links}. Proceedings of the 14th G\"okova Geometry-Topology Conference. (2007) 60-81. arXiv:math.0708.3249v2.

\bibitem{man} V. Manturov. {\it Minimal diagrams of classical and virtual links}. arXiv:math.GT/0501393.

\bibitem{mur} K. Murasugi. {\it On closed 3-braids}. Number 151 in Memoirs of the American Mathematical Society. American 
Mathematical Society, 1974. 

\bibitem{odd} P. Ozsv\'ath, J. Rasmussen, and S. Szab\'o. {\it Odd Khovanov homology.} arXiv:math.GT/0710.4300.

\bibitem{ras} J. Rasmussen. {\it Knot polynomials and knot homologies}. Geometry and topology of manifolds. Fields Inst. Commun. {\bf 47} 261-280. arXiv:math.GT/0504045.

\bibitem{sch} O. Schreier. {\it \"Uber die Gruppen $A^aB^b = 1$}. Abh. Math. Sem. Univ. Hamburg, {\bf 3} (1924), 167-169.

\bibitem{stos} M. Sto\v{s}i\'{c}. {\it Khovanov homology of torus links}. Journal of Topology and its Applications. (2008). arXiv:math/0606.5656.

\bibitem{tur} V. G. Turaev. {\it A simple proof of the Murasugi and Kauffman theorems on alternating links}. Enseign. Math. (2), 33(3-4):203-225, 1987. 869-884.

\bibitem{turn} P. Turner. {A spectral sequence for Khovanov homology with an application to $(3,q)$-torus links}. Algebraic and Geometric Topology. {\bf 8}, (2008), 869-884. arXiv:math.GT/0606369.

\bibitem{vir} O. Viro. {\it Remarks on the definition of the Khovanov homology}. arXiv:math.GT/0202199.

\bibitem{wat} L. Watson. {\it Surgery obstructions from Khovanov homology}. arXiv:math.GT/0807.1341.

\bibitem{wehrli} S. Wehrli. {\it A spanning tree model for Khovanov homology}. arXiv:math.GT/0409328.

\end{thebibliography}
\end{document}